\ifdefined\pdfoutput
  \pdfoutput=1
\fi
\documentclass[10pt]{article}
\usepackage[T1]{fontenc}
\usepackage{lmodern}
\usepackage[a4paper,margin=1in]{geometry}
\usepackage{amsmath,amssymb,amsthm,mathtools}
\usepackage[english]{babel}
\usepackage{microtype}
\usepackage[numbers,sort&compress]{natbib}
\usepackage{parskip}
\usepackage{booktabs}
\usepackage{xcolor}
\usepackage[colorlinks=true,linkcolor=blue!55!black,citecolor=blue!55!black,urlcolor=blue!55!black]{hyperref}
\usepackage[nameinlink,capitalise]{cleveref}
\newtheorem{theorem}{Theorem}[section]
\newtheorem{lemma}[theorem]{Lemma}
\newtheorem{proposition}[theorem]{Proposition}
\newtheorem{corollary}[theorem]{Corollary}
\theoremstyle{definition}
\newtheorem{remark}[theorem]{Remark}

\crefname{theorem}{Theorem}{Theorems}
\crefname{lemma}{Lemma}{Lemmas}
\crefname{proposition}{Proposition}{Propositions}
\crefname{corollary}{Corollary}{Corollaries}
\crefname{remark}{Remark}{Remarks}
\crefname{section}{\S}{Sections}

\AddToHook{env/theorem/begin}{\crefalias{section}{theorem}}
\AddToHook{env/lemma/begin}{\crefalias{theorem}{lemma}}
\AddToHook{env/proposition/begin}{\crefalias{theorem}{proposition}}
\AddToHook{env/corollary/begin}{\crefalias{theorem}{corollary}}
\AddToHook{env/remark/begin}{\crefalias{theorem}{remark}}

\newcommand{\Z}{\mathbb Z}
\newcommand{\floor}[1]{\left\lfloor #1\right\rfloor}
\newcommand{\bitand}{\mathbin{\&}}
\newcommand{\AND}{\operatorname{AND}}
\newcommand{\HW}{\operatorname{HW}}
\newcommand{\dotminus}{\mathbin{\dot{-}}}

\title{Short Recurrences and Straight-Line Programs for Primes}
\author{Joseph M. Shunia\thanks{Independent researcher, Ann Arbor, Michigan, United States. Email: \texttt{jshunia@gmail.com}.},
Mihai Prunescu\thanks{Research Center for Logic, Optimization and Security (LOS) (Faculty of Mathematics and Computer Science, University of Bucharest, Academiei 14, Bucharest (RO-010014), Romania), e-mail address: \texttt{mihai.prunescu@gmail.com}.} \thanks{ Simion Stoilow Institute of Mathematics of the Romanian Academy (Research unit 5, P. O. Box 1-764, Bucharest (RO-014700), Romania), e-mail address: \texttt{mihai.prunescu@imar.ro}.}}
\date{October 2026}

\hypersetup{
  pdftitle={Short Recurrences and Straight-Line Programs for Primes},
  pdfauthor={Joseph M. Shunia, Mihai Prunescu},
  pdfsubject={Short recurrences and straight-line programs for primes},
  pdfkeywords={prime recurrence, straight-line program, arithmetic term, binary encoding, packed sieve, n-th prime, 2-adic valuation, bitwise AND}
}

\begin{document}
\maketitle

\begin{abstract}\noindent
We give a short autonomous recurrence that generates every odd prime,
in increasing order, from the single initial value $8$. Its integer
state determines both its output and its next update: the lowest set bit
gives the next prime, while the highest set bit determines the sieve
period. The same state also determines the following prime gap. A related
recurrence represents the prime-counting function as the $2$-adic
valuation of a single integer.
\\[2mm] \noindent
Independently, a direct packing of a finite divisor sieve yields fixed
arithmetic constructions for the $n$-th prime, the prime-counting
function, and the twin-prime counting function. We express these
constructions as straight-line programs using addition, truncated
subtraction, multiplication, integer division, exponentiation, and
remainder. Counting operations with reuse of computed values gives
programs of lengths $166$, $132$, and $167$, respectively; one autonomous
recurrence update uses $99$ operations. Explicit arithmetic subroutines
remove all auxiliary operations, including bitwise AND. The programs
unfold into fixed arithmetic terms without variable-length sums,
products, or searches.
\end{abstract}

\begingroup\small
\par
\endgroup

\section{Introduction}
\label{sec:intro}

Starting from $T_0=8$, consider the recurrence
\begin{equation}
\label{eq:integer}
T_{k+1}
=
T_k\frac{2^{(\floor{\log_2 T_k}-1)\nu_2(T_k)}-1}
         {2^{\floor{\log_2 T_k}-1}-1}
-
\left(
T_k\frac{2^{(\floor{\log_2 T_k}-1)\nu_2(T_k)}-1}
         {2^{\floor{\log_2 T_k}-1}-1}
\;\bitand\;
\frac{2^{(\floor{\log_2 T_k}-1)\nu_2(T_k)}-1}
     {2^{\nu_2(T_k)}-1}
\right).
\end{equation}
Here $\nu_2(x)$ is the exponent of $2$ dividing a positive integer
$x$, and $\bitand$ denotes bitwise AND. Every quantity on the
right-hand side is determined by $T_k$ alone.

The first update gives
\[
T_1
=
8\frac{2^6-1}{2^2-1}
-
\left(
8\frac{2^6-1}{2^2-1}
\;\bitand\;
\frac{2^6-1}{2^3-1}
\right)
=168-(168\bitand9)=160.
\]
The next gives $T_2=2\,693\,408\,896$. The states grow rapidly,
but their $2$-adic valuations follow a familiar pattern:
\[
\nu_2(T_0),\nu_2(T_1),\nu_2(T_2),\ldots
=3,5,7,11,13,17,19,23,29,31,\ldots.
\]
The recurrence generates every odd prime, in order.

\begin{theorem}[One integer, successive primes]
\label{thm:main}
The recurrence \eqref{eq:integer} is well defined for every
$k\geq0$, and all divisions are exact. If $p_n$ denotes the
$n$-th prime, with $p_1=2$, then
\begin{equation}
\label{eq:valuation}
\nu_2(T_k)=p_{k+2}.
\end{equation}
\end{theorem}

No external prime list or separate counter is needed to advance
the recurrence. The state is \emph{self-describing}: its lowest
set bit identifies the next prime, while its highest set bit
determines the scale of the next update. It also reveals the
following prime gap. For $k\geq1$, the two lowest set bits occur
at consecutive prime positions, giving
\[
p_{k+3}-p_{k+2}
=\nu_2\!\left(\frac{T_k}{2^{\nu_2(T_k)}}-1\right).
\]
Thus dividing out the largest power of $2$ and subtracting $1$
turns the prime output into a prime-gap output
(\cref{cor:state-gap}).

The mechanism becomes visible in the binary expansion of $T_k$,
which encodes a finite reduced-residue pattern. The geometric
quotients in \eqref{eq:integer} repeat that pattern and mark
multiples of the current prime; bitwise AND and subtraction
remove the marked positions. Section~\ref{sec:integer-proof}
establishes this invariant and proves \cref{thm:main}.
It also identifies the encoded period as a primorial
(\cref{cor:state-period}) and derives a related recurrence,
indexed by the ordinary integers, whose state $Z_x$ satisfies
$\nu_2(Z_x)=\pi(x)$ (\cref{thm:pi-recurrence}). Here
$\pi(x)$ counts the primes not exceeding $x$. The same state
also counts twin candidates exactly: for $k\geq1$, a single
AND of $T_k$ with a shift of itself has Hamming weight
$\prod_{j=2}^{k+1}(p_j-2)$ (\cref{thm:twin-candidates}).

The second main construction returns $p_n$ directly from $n$.
It encodes a finite divisor sieve in a fixed straight-line
program, then selects the required prime by counting binary
carries. A shorter specialization returns $\pi(n)$ without the
rank-selection step, and one further AND of its packed prime
flags with a shift of themselves gives a fixed term for the
twin-prime counting function (\cref{thm:twin-packed}). These
constructions are independent of the recurrence and appear in
\cref{sec:nth-term}.
Finally, \cref{thm:and-term} in \cref{app:additional-terms}
expresses bitwise AND by an arithmetic term. Together with known terms for the remaining
auxiliary operations, this gives purely arithmetic forms of
the recurrences and fixed arithmetic terms for both $p_n$
and $\pi(n)$.

\paragraph{Related work.}
Formulas and recurrences for the primes have a long history.
Willans~\cite{Willans} gave formulas for $p_n$ built from
Wilson's theorem, whose length grows with $n$ through bounded
sums. Mills~\cite{Mills1947} proved that there is a real
constant $A$ such that $\floor{A^{3^k}}$ is prime for every
$k\geq1$. Fridman, Garbulsky, Glecer, Grime, and Tron
Florentin~\cite{Fridman2019} gave a constant $f_1$ for which
the recurrence
$f_{n+1}=\floor{f_n}\bigl(f_n-\floor{f_n}+1\bigr)$ satisfies
$\floor{f_n}=p_n$. In both cases the constant must be computed
to high precision, which in effect requires the primes in
advance. Rowland~\cite{Rowland} gave the recurrence
$a(n)=a(n-1)+\gcd(n,a(n-1))$, whose first differences are
all $1$ or prime; the primes appear with repetitions and
interspersed with $1$s.

By contrast, \eqref{eq:integer} uses only exact integer
arithmetic, starts from the integer $8$, and needs no
precomputed constant. It is autonomous and self-describing,
and it outputs each odd prime exactly once, in increasing
order. The price is size. By \cref{cor:state-period}, the
bit length of $T_k$ is $p_1p_2\cdots p_{k+1}+2$, and $T_k$ is
a binary copy of a primorial wheel, the reduced-residue cycle
whose gap structure is studied by Holt and
Rudd~\cite{HoltRudd2014}. The arithmetic-term background for
the second construction is recalled in \cref{sec:arith-context}.

\subsection{Straight-line programs and arithmetic terms}
\label{sec:arith-context}

We use the counting convention of Dumitru and Prunescu
\cite{DumitruPrunescuPell}. Our primitive operations on
$\mathbb N=\{0,1,2,\ldots\}$ are
\[
a+b,\qquad a\mathbin{\dot{-}}b=\max(a-b,0),\qquad ab,\qquad
\floor{a/b},\qquad a^b,\qquad a\bmod b,
\]
with positive divisors and moduli. Each displayed primitive costs one
operation, including a multiplication by a fixed constant and an
exponentiation with fixed exponent. Constants, assignments, and reuse of
already computed values cost nothing. A straight-line program is a fixed
finite list of such assignments, without a loop or an input-dependent
branch. Several operations may be grouped in one displayed assignment;
our counts charge every primitive in that assignment. A named subroutine
means insertion of its fixed instruction list, with local registers.
All counts are upper bounds supplied by the displayed programs, not claims
of optimality.

The counts therefore measure arithmetic circuits with shared intermediate
values. An arithmetic term is obtained by substituting every assignment
into its later uses; this can duplicate subexpressions. Remainder can also
be removed through
\[
a\bmod b=a\mathbin{\dot{-}}b\floor{a/b},
\]
which uses three operations in the five-operation arithmetic-term basis.
Every subtraction used for a numerical identity below is nonnegative on
its stated domain. When convenient, we write ordinary subtraction for
these exact instances of truncated subtraction.

One application of the recurrence is a straight-line program. Iterating
it to produce arbitrarily many primes is a loop, and is not itself a
single fixed straight-line program. In contrast, the direct programs in
Section~\ref{sec:nth-term} take $n$ as input and have a number of
instructions independent of $n$. Exponentiation is one arithmetic
operation in this convention; the cost of manipulating its potentially
huge output is considered separately from program length.

Mazzanti~\cite{Mazzanti} proved that every Kalm\'ar elementary function can
be represented by an arithmetic term; see also Marchenkov's work on
superposition bases~\cite{Marchenkov,Marchenkov2007}. Since $n\mapsto p_n$
is Kalm\'ar elementary, an arithmetic term for $p_n$ exists in principle.
Mazzanti's general method, however, counts solutions of Diophantine
descriptions by a hypercube construction, and the resulting terms are
extremely long. The interest therefore lies in explicit, comparatively short
constructions. Such terms are known for a range of number-theoretic
functions including $\nu_2$ and
$\floor{\log_2}$~\cite{PrunescuSaurasNumberTheory}. At the level of the whole Kalm\'ar
class, addition, remainder, and base-two exponentiation already suffice
under substitution~\cite{PrunescuSaurasShuniaBasis}. For the primes,
Prunescu and Shunia constructed explicit terms for $p_n$ and $\pi(n)$ by
applying the hypercube method to a number-theoretic
predicate~\cite{PrunescuShuniaPrime}, and asked whether a simpler
construction avoiding that method exists. The packed sieve below gives one;
the two approaches are compared at the end of \cref{sec:nth-strict}.

The arithmetic subroutines used in our operation counts are displayed
in Appendix~\ref{sec:slp-aux}. They represent factorial, binomial
coefficient, greatest common divisor, valuation, Hamming weight, binary
logarithm, and bitwise AND in the same six-operation basis. The underlying
formulas are due to
\cite{PrunescuSaurasFactorial,Robinson1952,PrunescuShuniaGCD,PrunescuSaurasNumberTheory}
or proved in the appendix. We write $\HW(a)$ for the number of ones in the
binary expansion of $a$, and $\AND(a,b)$ for binary bitwise conjunction.
These names are subroutine calls, not additional unit-cost primitives. Here we show only four examples:
\begin{align*}
\gcd(a,b) &= \left ( \left \lfloor{\frac{5^{ab(ab + a+b)}}{(5^{a^2b}-1)(5^{b^2a}-1)}} \right \rfloor \bmod 5^{ab} \right ) \dot{-} 1, \\
\nu_2(n) &= \left \lfloor{\frac{\gcd(n, 2^n)^{n+1} \bmod (2^{n+1}-1)^2}{2^{n+1}-1}} \right \rfloor , \\
\binom{n}{k} &= \left \lfloor{\frac{(2^n+1)^n}{2^{nk}}} \right \rfloor \bmod 2^n , \\
n! &= \floor{ (8^{n^2})^n \Big/ \binom{8^{n^2}}{n} } , \\
\HW(n) &= \nu_2 \left( \binom{2n}{n} \right).
\end{align*}

The constructions and their sieve invariants are developed in the main
text. Appendix~\ref{app:additional-terms} supplies the arithmetic
subroutines, and Appendix~\ref{app:programs} gives the complete register
schedules with shared intermediate values. Their lengths are summarized
in Table~\ref{tab:program-lengths}.

\begin{table}[htbp]
\centering
\begin{tabular}{@{}llr@{}}
\toprule
Input & Output & Operations\\
\midrule
$T_k$ & $T_{k+1}$ and $p_{k+2}$ & $99$\\
$(x,Z_x)$ & $Z_{x+1}$ & $129$\\
$n\geq1$ & $p_n$ & $166$\\
$n\geq1$ & $\pi(n)$ & $132$\\
$n\geq1$ & $\pi_2(n)$ & $167$\\
\bottomrule
\end{tabular}
\caption{Lengths of the displayed straight-line programs in the
six-operation basis. All auxiliary calls are expanded; constants and
reuse are free. Here $\pi_2(n)$ counts pairs $(p,p+2)$ with $p+2\leq n$.}
\label{tab:program-lengths}
\end{table}

Table~\ref{tab:program-lengths} summarizes the explicit instruction
counts established by the schedules in Appendix~\ref{app:programs}. They use the arithmetic subroutines of
Appendix~\ref{sec:slp-aux}; no auxiliary operation is counted as a
unit-cost primitive. The first two rows concern one recurrence update.
The last three are fixed programs on the single input $n$.

We use geometric sums with a \emph{number-of-terms} parameter:
\begin{equation}
\label{eq:twin-packed-geometric-terms}
\begin{aligned}
G_0(x,m)&=\sum_{j=0}^{m-1}x^j=\frac{x^m-1}{x-1},\\
G_1(x,m)&=\sum_{j=0}^{m-1}jx^j
=\frac{\bigl((m-1)x^{m+1}+x\bigr)-mx^m}{(x-1)^2}.
\end{aligned}
\end{equation}
Here $x\geq2$, with $m\geq0$ for $G_0$ and $m\geq1$ for $G_1$.
In particular, $G_0(x,0)=0$. Both quotients are exact. The grouping in the
numerator of $G_1$ ensures nonnegative subtraction. The sums describe the
outputs; their evaluation uses only the displayed closed forms.
The packed program in Appendix~\ref{sec:packed-slp} reuses powers
already computed for its other registers instead of evaluating each occurrence of $G_0,G_1$ separately.

\section{Proof of the integer recurrence}
\label{sec:integer-proof}

To make the structure of \eqref{eq:integer} explicit, fix $k\geq0$
and introduce auxiliary notation for a single update:
\begin{equation}
\label{eq:parameters}
L=\floor{\log_2 T_k}-1,\qquad
M=2^{L\nu_2(T_k)}-1,\qquad
U=T_k\frac{M}{2^L-1},\qquad
V=\frac{M}{2^{\nu_2(T_k)}-1}.
\end{equation}
The recurrence then takes the equivalent form
\begin{equation}
\label{eq:parameters_integer}
T_{k+1}=U-(U\bitand V).
\end{equation}
These quantities are recomputed from $T_k$ at each step;
the state remains a single integer. Bitwise AND need not be
admitted as a primitive: with the arithmetic term $\AND$ of
\cref{thm:and-term}, proved in \cref{app:additional-terms}, the
update becomes
\begin{equation}
\label{eq:integer-arithmetic}
T_{k+1}=U-\AND(U,V).
\end{equation}

These shared quantities give a fixed straight-line update. Expanding its
arithmetic subroutines and retaining values already computed gives the
$99$-operation schedule in Appendix~\ref{sec:slp-update}. The proof below
identifies the binary pattern carried by the state and explains why the
update advances from one prime to the next.

Let $p_n$ denote the $n$-th prime, and define
\begin{equation}
\label{eq:Lk}
\Pi_k=\prod_{j=2}^{k+1}p_j,\qquad L_k=2\Pi_k,
\end{equation}
where an empty product is $1$, so $\Pi_0=1$ and $L_0=2$.
These quantities describe the invariant used in the proof;
they are not additional inputs to the recurrence.

\begin{lemma}[The binary invariant]
\label{lem:invariant}
For each $k\geq0$, put
\begin{equation}
\label{eq:wheel-state}
\widetilde T_k
=\sum_{\substack{2\leq m\leq L_k+1\\\gcd(m,L_k)=1}}2^m.
\end{equation}
Then
\begin{equation}
\label{eq:self}
\nu_2(\widetilde T_k)=p_{k+2},\qquad
\floor{\log_2\widetilde T_k}-1=L_k.
\end{equation}
\end{lemma}

\begin{proof}
The least integer greater than $1$ and coprime to $L_k$ is
$p_{k+2}$. Every smaller integer greater than $1$ has a prime
divisor among $p_1,\ldots,p_{k+1}$, whereas $p_{k+2}$ does not.
Since $L_k+1$ is itself coprime to $L_k$, this least integer
lies in the summation range in \eqref{eq:wheel-state}.

The lowest and highest occupied positions are therefore
$p_{k+2}$ and $L_k+1$, respectively. Each coefficient is
$0$ or $1$, so the sum is a binary expansion without carries.
The two identities follow.
\end{proof}

\begin{proof}[Proof of \cref{thm:main}]
We prove by induction that the recurrence is well defined and
$T_k=\widetilde T_k$. For $k=0$, the only integer in $[2,3]$
coprime to $2$ is $3$, so $\widetilde T_0=2^3=8=T_0$.

Assume $T_k=\widetilde T_k$. For this induction step, write
$p=\nu_2(T_k)$. By \cref{lem:invariant},
\[
p=p_{k+2},\qquad L=L_k,
\]
and hence $M=2^{Lp}-1$. In particular, $L\geq2$ and $p\geq3$,
so the denominators in \eqref{eq:parameters} are nonzero.
The geometric identity
\[
\frac{2^{Lp}-1}{2^L-1}=\sum_{j=0}^{p-1}2^{jL}
\]
shows that multiplication by this quotient places translated
copies of the current binary pattern in the disjoint intervals
$[jL+2,(j+1)L+1]$, for $0\leq j<p$. Translation by $jL$
preserves coprimality to $L$. Consequently no carries occur, and
\[
U
=T_k\frac{2^{Lp}-1}{2^L-1}
=\sum_{\substack{2\leq m\leq Lp+1\\\gcd(m,L)=1}}2^m.
\]
Similarly,
\[
V=\frac{2^{Lp}-1}{2^p-1}=\sum_{j=0}^{L-1}2^{jp}
\]
has set bits at the multiples of $p$ below $Lp$. The only
multiple of $p$ in $[2,Lp+1]$ not covered by this mask is $Lp$,
which is already absent from $U$ because $\gcd(Lp,L)=L>1$.
Thus $U\bitand V$ consists exactly of the set bits of $U$
whose positions are divisible by $p$.

Every set bit of $U\bitand V$ is also a set bit of $U$, so
the subtraction in \eqref{eq:parameters_integer} clears those
positions without borrowing. Therefore
\[
T_{k+1}
=\sum_{\substack{2\leq m\leq Lp+1\\\gcd(m,Lp)=1}}2^m
=\widetilde T_{k+1},
\]
because $Lp=L_{k+1}$. The geometric identities establish the
exactness of both divisions, and the bit at position $Lp+1$
ensures that the new state is positive. This completes the
induction. The prime-output identity follows from
\cref{lem:invariant}.
\end{proof}

\begin{corollary}[The period encoded by the state]
\label{cor:state-period}
For every $k\geq0$,
\[
\floor{\log_2 T_k}-1=\prod_{j=1}^{k+1}p_j.
\]
Consequently,
\[
\floor{\log_2 T_{k+1}}-1
=\bigl(\floor{\log_2 T_k}-1\bigr)\nu_2(T_k).
\]
\end{corollary}

\begin{proof}
The first identity is \eqref{eq:self}, since
$L_k=\prod_{j=1}^{k+1}p_j$. The second follows from
$L_{k+1}=L_kp_{k+2}$ and $p_{k+2}=\nu_2(T_k)$.
\end{proof}

\begin{corollary}[Prime gaps from a single state]
\label{cor:state-gap}
For every $k\geq1$, the two lowest set bits of $T_k$ occur
at positions $p_{k+2}$ and $p_{k+3}$. In particular,
\[
\nu_2\!\left(T_k-2^{\nu_2(T_k)}\right)=p_{k+3},
\]
and
\begin{equation}
\label{eq:state-gap}
p_{k+3}-p_{k+2}
=\nu_2\!\left(\frac{T_k}{2^{\nu_2(T_k)}}-1\right).
\end{equation}
\end{corollary}

\begin{proof}
The set bits of $T_k$ occur precisely at the positions in
$[2,L_k+1]$ that are coprime to $L_k$.
Both $p_{k+2}$ and $p_{k+3}$ lie in this interval.
For $k=1$, this follows directly from $T_1=2^5+2^7$.
For $k\geq2$, we have $L_k\geq6p_{k+1}$, so Bertrand's
postulate gives
\[
p_{k+3}<2p_{k+2}<4p_{k+1}\leq L_k.
\]
Every composite integer coprime to $L_k$ has all its prime
factors at least $p_{k+2}$ and is therefore at least
$p_{k+2}^2$. Since
$p_{k+3}<2p_{k+2}<p_{k+2}^2$, no composite position can
intervene between the two stated prime positions.
They are therefore the two lowest set bits.

Subtracting $2^{\nu_2(T_k)}$ removes the lowest set bit,
giving the first identity. Dividing the remaining integer
by $2^{\nu_2(T_k)}$ gives \eqref{eq:state-gap}.
\end{proof}

\begin{remark}
The initial state $T_0=8$ has only one set bit. A uniform
version of \eqref{eq:state-gap}, valid for every $k\geq0$, is
\[
p_{k+3}-p_{k+2}
=\nu_2\!\left(
\frac{T_k}{2^{\nu_2(T_k)}}-1
+2^{\floor{\log_2 T_k}-1}
\right).
\]
For $k\geq1$, the preceding proof gives $p_{k+3}\leq L_k+1$,
and hence
\[
p_{k+3}-p_{k+2}\leq L_k+1-p_{k+2}<L_k.
\]
The added power of $2$ has exponent $L_k$, so it does not
change the valuation. For $k=0$, the expression inside the
valuation is $4$, giving the gap $5-3=2$.
\end{remark}

\subsection{A one-integer recurrence for the prime-counting function}
\label{sec:pi-recurrence}

Let $\Phi(T)$ denote one application of the update in
\eqref{eq:parameters} and \eqref{eq:parameters_integer}:
the quantities $L,M,U,V$ are computed from $T$, and
$\Phi(T)=U-(U\bitand V)$. On the states of \cref{thm:main},
$\Phi(T_k)=T_{k+1}$. We use $\pi(x)$ for the number of primes
not exceeding $x$.

\begin{theorem}[Prime counting as a valuation]
\label{thm:pi-recurrence}
Set $Z_2=34$. For each integer $x\geq2$, recover
\begin{equation}
\label{eq:pi-recurrence-decode}
c_x=\nu_2(Z_x),
\qquad
W_x=\frac{Z_x/2^{c_x}-1}{2},
\end{equation}
and define
\begin{equation}
\label{eq:pi-recurrence-indicator}
\varepsilon_x
=\floor{\frac{1}{(\nu_2(W_x)-x-1)^2+1}}.
\end{equation}
Put
\begin{equation}
\label{eq:pi-recurrence-wheel}
W_x^+=W_x+\varepsilon_x\bigl(\Phi(W_x)-W_x\bigr)
\end{equation}
and
\begin{equation}
\label{eq:pi-recurrence-state}
Z_{x+1}=2^{c_x+\varepsilon_x}(2W_x^++1).
\end{equation}
Then the recurrence is well defined, all divisions outside
the floor are exact, and, for every integer $x\geq2$,
\begin{equation}
\label{eq:pi-recurrence-invariant}
\nu_2(Z_x)=\pi(x),
\qquad
W_x=T_{\pi(x)-1}.
\end{equation}
Thus the prime-counting function is the $2$-adic valuation
of a one-integer recurrence.
\end{theorem}

\begin{proof}
Since $Z_2=34=2(2T_0+1)$, decoding gives $c_2=1=\pi(2)$
and $W_2=T_0$, so the assertion holds at $x=2$.

Assume \eqref{eq:pi-recurrence-invariant} holds at $x$.
By \cref{thm:main},
\[
\nu_2(W_x)=\nu_2(T_{\pi(x)-1})=p_{\pi(x)+1},
\]
the least prime strictly greater than $x$. The reciprocal
in \eqref{eq:pi-recurrence-indicator} is $1$ when
$x+1=\nu_2(W_x)$ and lies strictly between $0$ and $1$
otherwise. Hence $\varepsilon_x=1$ exactly when $x+1$ is
prime, and $\varepsilon_x=0$ otherwise. It follows that
\[
\pi(x+1)=\pi(x)+\varepsilon_x.
\]
The update $\Phi(W_x)$ is well defined by \cref{thm:main}.
If $\varepsilon_x=0$, then $W_x^+=W_x$; if
$\varepsilon_x=1$, then $W_x^+=\Phi(W_x)=T_{\pi(x)}$.
In either case,
\[
W_x^+=T_{\pi(x+1)-1}.
\]
Equation \eqref{eq:pi-recurrence-state} therefore gives
\[
Z_{x+1}
=2^{\pi(x+1)}\bigl(2T_{\pi(x+1)-1}+1\bigr).
\]
The factor in parentheses is positive and odd, so
$\nu_2(Z_{x+1})=\pi(x+1)$. Applying
\eqref{eq:pi-recurrence-decode} then recovers
$W_{x+1}=T_{\pi(x+1)-1}$, with both divisions exact.
The divisions within $\Phi$ are exact by \cref{thm:main}.
This completes the induction.
\end{proof}

\begin{remark}
The recurrence in \cref{thm:pi-recurrence} uses one stored
integer, but its update depends explicitly on the index $x$.
The recurrence of \cref{thm:main} is autonomous: its update
depends only on its state. In the prime-counting recurrence,
the wheel state advances exactly when the next integer is
prime, while the exponent of $2$ records the current prime
count, including the initial prime $2$. Substituting
\cref{thm:and-term} into $\Phi$ removes bitwise AND as a
primitive operation.
\end{remark}

\subsection{Twin primes and pairs of surviving positions}
\label{sec:twin-primes}

The binary state also records pairs of surviving integers separated by
$2$. We give exact formulas for their number and a criterion under which
such a pair consists of primes. The arguments use the binary-state
invariant established in Section~\ref{sec:integer-proof}.
For this subsection, write
\[
L_k=\prod_{j=1}^{k+1}p_j,
\qquad q_k=\nu_2(T_k)=p_{k+2}.
\]
Thus $L_k$ agrees with the notation in \eqref{eq:Lk}, and
\cref{lem:invariant} together with the proof of \cref{thm:main} gives
\begin{equation}
\label{eq:twin-state-invariant}
T_k=\sum_{\substack{2\leq m\leq L_k+1\\\gcd(m,L_k)=1}}2^m.
\end{equation}
As elsewhere, $L_k$ and $q_k$ are auxiliary descriptions of the state,
not additional inputs to the recurrence.

\begin{proposition}[A residue test for a twin-prime output]
\label{prop:twin-output}
For every $k\geq1$,
\begin{equation}
\label{eq:twin-output}
p_{k+3}-p_{k+2}=2
\quad\Longleftrightarrow\quad
\frac{T_k}{2^{\nu_2(T_k)}}\equiv5\pmod8.
\end{equation}
\end{proposition}

\begin{proof}
By \cref{cor:state-gap},
\[
p_{k+3}-p_{k+2}
=\nu_2\!\left(\frac{T_k}{2^{\nu_2(T_k)}}-1\right).
\]
The positive integer inside the valuation has valuation $2$ precisely
when it is congruent to $4$ modulo $8$. This is equivalent to the
congruence in \eqref{eq:twin-output}.
\end{proof}

For $k\geq1$, define the integer
\begin{equation}
\label{eq:twin-mask}
A_k=T_k\bitand\floor{\frac{T_k}{4}}.
\end{equation}
Division by $4$, followed by the floor, shifts the binary expansion
two positions to the right. Consequently, the bit in position $m$
of $A_k$ is set exactly when positions $m$ and $m+2$ are both set
in $T_k$. We call such a pair a \emph{twin candidate} at stage $k$.
Let $\HW(a)$ denote the number of set bits in the binary expansion
of a nonnegative integer $a$.

\begin{theorem}[Exact enumeration of twin candidates]
\label{thm:twin-candidates}
For every $k\geq1$,
\begin{equation}
\label{eq:twin-mask-invariant}
A_k=
\sum_{\substack{2\leq m\leq L_k-1\\
\gcd(m(m+2),L_k)=1}}2^m.
\end{equation}
In particular, $A_k>0$, and
\begin{equation}
\label{eq:twin-candidate-count}
\HW(A_k)=\prod_{j=2}^{k+1}(p_j-2).
\end{equation}
The counts therefore satisfy
\begin{equation}
\label{eq:twin-candidate-growth}
\HW(A_{k+1})
=\bigl(\nu_2(T_k)-2\bigr)\HW(A_k).
\end{equation}
\end{theorem}

\begin{proof}
The description of the set bits of $A_k$, together with
\eqref{eq:twin-state-invariant}, gives
\eqref{eq:twin-mask-invariant}. The position $L_k-1$ occurs in this
sum, since both $L_k-1$ and $L_k+1$ are coprime to $L_k$.
Hence $A_k>0$.

Because $k\geq1$, the integer $L_k$ is divisible by $6$.
Among the residue classes modulo $L_k$, the class $0$ fails the
coprimality condition, and the class $1$ fails it modulo $3$.
Thus the interval $[2,L_k-1]$ contains exactly one representative
of every class satisfying $\gcd(m(m+2),L_k)=1$.

Modulo $2$, precisely one class is allowed. Modulo each odd prime
$p_j\mid L_k$, precisely the two distinct classes $0$ and $-2$
are forbidden, leaving $p_j-2$ allowed classes. Since $L_k$ is
squarefree, the Chinese remainder theorem gives
\eqref{eq:twin-candidate-count}. Adjoining the next prime
$q_k=\nu_2(T_k)$ to the product gives
\eqref{eq:twin-candidate-growth}.

Equivalently, each admissible class modulo $L_k$ has $q_k$ lifts
modulo $L_kq_k$. Exactly one lift has its first member divisible
by $q_k$, and exactly one has its second member divisible by $q_k$.
These lifts are distinct because $q_k\nmid2$, so exactly $q_k-2$
survive.
\end{proof}

The product in \eqref{eq:twin-candidate-count} is the classical
count of gaps of size $2$ in a primorial reduced-residue cycle;
see Holt and Rudd~\cite{HoltRudd2014}. Here it is the Hamming
weight of the single integer \eqref{eq:twin-mask}. For example,
\[
A_1=2^5,
\qquad
A_2=2^{11}+2^{17}+2^{29}.
\]
These record the pairs $(5,7)$ and
$(11,13),(17,19),(29,31)$, respectively. Later states may also
contain candidates with composite members: $A_3$ contains the
pair $(209,211)$, although $209=11\cdot19$.

\begin{proposition}[Certification below the square threshold]
\label{prop:twin-square}
Let $k\geq1$. If the bit in position $m$ of $A_k$ is set and
$m+2<q_k^2$, then $m$ and $m+2$ are twin primes. In particular,
\begin{equation}
\label{eq:twin-square-criterion}
\nu_2(A_k)+2<\nu_2(T_k)^2
\end{equation}
implies that $\nu_2(A_k)$ and $\nu_2(A_k)+2$ are twin primes.
\end{proposition}

\begin{proof}
Every integer greater than $1$ and coprime to $L_k$ has all its
prime factors at least $q_k$. If such an integer is composite,
it is therefore at least $q_k^2$. By
\eqref{eq:twin-mask-invariant}, both $m$ and $m+2$ are coprime
to $L_k$ and greater than $1$. The stated upper bound makes
both prime. The final assertion follows by taking
$m=\nu_2(A_k)$, which is defined because $A_k>0$.
\end{proof}

The next corollary restates the twin-prime conjecture in terms of the
states $T_k$. It is a reformulation only and gives no progress toward the
conjecture.

\begin{corollary}[Reformulations of the twin-prime conjecture]
\label{cor:twin-infinitude}
The following statements are equivalent:
\begin{enumerate}
\item There are infinitely many twin-prime pairs.
\item The congruence
\[
\frac{T_k}{2^{\nu_2(T_k)}}\equiv5\pmod8
\]
holds for infinitely many $k\geq1$.
\item The inequality
\[
\nu_2\!\left(T_k\bitand\floor{\frac{T_k}{4}}\right)+2
<\nu_2(T_k)^2
\]
holds for infinitely many $k\geq1$.
\end{enumerate}
\end{corollary}

\begin{proof}
The equivalence of the first two statements follows from
\cref{prop:twin-output}, since the values $p_{k+2}$ with $k\geq1$
run through all primes at least $5$.

Suppose the second statement holds at an index $k$. Then
$q_k$ and $q_k+2$ are the two lowest occupied positions of $T_k$,
so $\nu_2(A_k)=q_k$. Since $q_k\geq5$,
\[
\nu_2(A_k)+2=q_k+2<q_k^2.
\]
The third statement therefore follows from the second.

Conversely, each index satisfying the third statement gives a
twin-prime pair by \cref{prop:twin-square}. Every occupied
position of $A_k$ is also occupied in $T_k$, so
\[
\nu_2(A_k)\geq\nu_2(T_k)=p_{k+2}.
\]
Along any infinite set of indices these lower bounds tend to
infinity. The resulting twin-prime pairs are therefore unbounded,
which proves the first statement.
\end{proof}

\begin{remark}
Equation~\eqref{eq:twin-candidate-growth} counts surviving pairs
over a complete period of length $L_k$. The inequality
\eqref{eq:twin-square-criterion} concerns the position of the
first such pair. Establishing that inequality for infinitely
many indices requires control of these positions beyond the
total count; the counting formula alone does not establish
twin-prime infinitude.
\end{remark}

\section{Packed divisor sieves and prime extraction}
\label{sec:nth-term}

The recurrence above is short because it carries a large state from one
prime to the next. A fixed straight-line program is short in a different sense:
its input is the rank $n$, and its instruction list does not grow
with $n$. We now give such a program for the $n$-th prime $p_n$, where
\[
p_1=2,\qquad p_2=3,\qquad p_3=5,\ldots.
\]
This construction is independent of \cref{thm:main}. It packs a finite
divisor sieve into large integers and then selects the desired rank by
counting binary carries. The form below uses the quadratic cutoff $n(n+1)$, tests
divisibility only by integers at most $n$, and needs two AND operations
followed by one binomial valuation.

\subsection{The packed construction}
\label{sec:nth-packed}

Let $n\geq 1$. To keep the roles of the parameters visible, write
\begin{align}
C_n&=n(n+1) &&\text{(candidate cutoff)},\nonumber\\
F_n&=(n+1)! &&\text{(common divisor multiple)},\nonumber\\
h_n&=(n+2)^2 &&\text{(field-width exponent)},
\label{eq:nth-parameters-a}\\
D_n&=2^{h_n},\qquad
H_n=2^{n+1},\qquad
B_n=D_nH_n,\nonumber\\
A_n&=B_n^{C_n+1},\qquad
E_n=A_n^{n+1},\qquad
\mathcal U_n=G_0(B_n,C_n+1).
\label{eq:nth-parameters-b}
\end{align}
Here $C_n$ is the largest candidate that will be represented, $B_n$ is the
base of one candidate field, and
$\mathcal U_n=1+B_n+\cdots+B_n^{C_n}$ is the corresponding repunit.

Define the packed reciprocal integer
\begin{equation}
\label{eq:nth-C}
\mathcal R_n=
\frac{
E_n^2+1-\bigl((A_n-1)^{E_nF_n}\bmod E_n^2\bigr)
}{E_n}
-F_nA_n,
\end{equation}
and set
\begin{align}
\kappa_n&=\floor{\frac{D_n+F_n-1}{F_n}},
\label{eq:nth-kappa}\\
\mathcal Q_n&=
\frac{
\bigl(\kappa_n\mathcal R_n G_1(B_n,C_n+1)\bigr)
\bitand
\bigl(D_n(H_n-1)G_0(B_n,(C_n+1)(n+1))\bigr)
}{D_n},
\label{eq:nth-Q}\\
\mathcal W_n&=\mathcal Q_n\bmod(A_n-1),
\label{eq:nth-W}\\
\Delta_n&=
\bigl((A_n-B_n+1)\mathcal W_n\bmod A_n\bigr)
-B_n^2G_0(B_n,n-1),
\label{eq:nth-T}\\
\mathcal P_n&=\mathcal U_n-1-B_n-
\frac{
\bigl(\Delta_n+(H_n-1)\mathcal U_n\bigr)\bitand(H_n\mathcal U_n)
}{H_n},
\label{eq:nth-P}\\
\mathcal S_n&=(\mathcal P_n\mathcal U_n)\bmod A_n.
\label{eq:nth-S}
\end{align}
Finally define
\begin{equation}
\label{eq:nth-answer}
\widehat{p}_n
=C_n+1-
\nu_2\!\binom{
\mathcal S_n+(2H_n-n)\mathcal U_n
}{H_n\mathcal U_n}.
\end{equation}

\begin{theorem}[A packed straight-line construction for $p_n$]
\label{thm:nth-packed}
For every $n\geq 1$, all quantities in
\eqref{eq:nth-parameters-a}--\eqref{eq:nth-answer} are defined, every
displayed quotient outside a floor is exact, and
\[
\widehat{p}_n=p_n.
\]
The packed construction uses two bitwise $\AND$ operations, one factorial, and
one $2$-adic valuation of a binomial coefficient. It takes neither a
preceding prime nor a search, variable-length sum, or variable-length
product as an input operation.
\end{theorem}

The program in Appendix~\ref{sec:packed-slp} evaluates these formulas
with shared intermediate values. Its fully arithmetic implementation uses
$166$ operations for $p_n$; the specializations for $\pi(n)$ and
$\pi_2(n)$ use $132$ and $167$, respectively. We first explain why the
packed integers encode the required sieve and rank information.

The proof is split into three elementary pieces: parameter bounds, reciprocal packing, and the packed divisor sieve. The only external estimate
on primes is the elementary bound of Ruiz and Sondow \cite{RuizSondow}.

\begin{lemma}[Parameter bounds]
\label{lem:nth-bounds}
The parameters above satisfy
\begin{equation}
\label{eq:nth-bounds}
p_n\leq C_n<(n+1)^2,
\qquad C_n<H_n,
\qquad C_nF_n<D_n,
\qquad F_n<D_n<B_n<A_n.
\end{equation}
In particular $D_n\geq 8$, $B_n>4H_n$, and $\nu_2(A_n)\geq n+1$.
\end{lemma}

\begin{proof}
For $n>1$, Ruiz and Sondow prove $p_n<2n\log n+2$
\cite[Lemma~1]{RuizSondow}. The function
\[
x+1-2\log x-\frac2x
\]
is positive at $x=2$ and has positive derivative $1-2/x+2/x^2$ for
$x\geq 2$. Hence $p_n<n(n+1)=C_n$ for $n\geq 2$, while $p_1=C_1=2$. The
inequality $C_n<(n+1)^2$ is immediate.

The inequality $2^{n+1}>n(n+1)$ holds for $n=1,2$ and propagates by
induction, so $C_n<H_n$. Also $n+1\leq 2^n$ gives
\[
C_nF_n=n(n+1)(n+1)!
\leq(n+1)^{n+3}
\leq 2^{n(n+3)}
<2^{(n+2)^2}=D_n.
\]
The remaining inequalities follow from the definitions. Since
$B_n=2^{h_n+n+1}$,
\[
\nu_2(A_n)=(h_n+n+1)(C_n+1)\geq n+1.
\]
\end{proof}

\begin{lemma}[Direct reciprocal packing]
\label{lem:reciprocal-packing}
Let $m\geq 1$, let $F_0$ be a positive common multiple of $1,\ldots,m+1$,
and let $a=2^b$ with $b\geq m+1$ and $a>F_0$. Put $e=a^{m+1}$. Then
\begin{equation}
\label{eq:reciprocal-packing}
\frac{
e^2+1-\bigl((a-1)^{eF_0}\bmod e^2\bigr)
}{e}
=
\sum_{j=1}^{m}\frac{F_0}{j}a^j.
\end{equation}
The sum describes the packed result; the left side computes it with one
exponentiation and one remainder.
\end{lemma}

\begin{proof}
Write $\rho=eF_0$ and $\ell=b(m+1)=\nu_2(e)$. For $1\leq j\leq m$,
\[
\frac1e\binom{\rho}{j}
=
\frac{F_0}{j}
\prod_{u=1}^{j-1}\left(\frac{eF_0}{u}-1\right)
\equiv
(-1)^{j-1}\frac{F_0}{j}
\pmod e,
\]
because every $u\leq j$ divides $F_0$. Consequently the $j$-th nonconstant
term of $(1-a)^\rho$ is congruent to
\[
-e\frac{F_0}{j}a^j\pmod{e^2}.
\]

For every $1\leq j\leq\rho$, the identity
$j\binom{\rho}{j}=\rho\binom{\rho-1}{j-1}$ gives
\[
\nu_2\!\left(\binom{\rho}{j}a^j\right)
\geq
\ell+\nu_2(F_0)-\nu_2(j)+bj.
\]
At $j=m+1$ the right side is at least $2\ell$, because $m+1$ divides $F_0$.
If $j=m+1+t$ with $t\geq 1$, then
\[
\nu_2(j)\leq j-1=m+t\leq(m+1)t\leq bt,
\]
so the same lower bound $2\ell$ follows. Every term with $j\geq m+1$
therefore vanishes modulo $e^2$. Since $\rho$ is even,
\[
(a-1)^\rho=(1-a)^\rho
\equiv
1-e\sum_{j=1}^{m}\frac{F_0}{j}a^j
\pmod{e^2}.
\]
Set
\[
c=\sum_{j=1}^{m}\frac{F_0}{j}a^j.
\]
Because $0<F_0/j<a$, this is a base-$a$ integer whose largest occupied
position is $m$, so $0<c<e=a^{m+1}$. Therefore
\[
(a-1)^\rho\bmod e^2=e^2+1-ec,
\]
and solving for $c$ proves \eqref{eq:reciprocal-packing}.
\end{proof}

Apply \cref{lem:reciprocal-packing} with $m=n$, $F_0=F_n$, $a=A_n$, and
$e=E_n$, then subtract the $j=1$ term. This gives
\begin{equation}
\label{eq:nth-C-meaning}
\mathcal R_n=\sum_{j=2}^{n}\frac{F_n}{j}A_n^j.
\end{equation}
For $n=1$ the sum is empty and $\mathcal R_1=0$.

\begin{lemma}[The short divisor sieve]
\label{lem:nth-sieve}
The quantities in \eqref{eq:nth-Q}--\eqref{eq:nth-S} satisfy
\begin{equation}
\label{eq:nth-Q-meaning}
\mathcal Q_n
=\sum_{j=2}^{n}\sum_{i=0}^{C_n}
\floor{\frac ij}B_n^iA_n^j.
\end{equation}
If
\[
t_i=\#\{j:2\leq j\leq n,\ j<i,\ j\mid i\},
\qquad t_0=t_1=0,
\]
then
\begin{equation}
\label{eq:nth-T-meaning}
\Delta_n=\sum_{i=0}^{C_n}t_iB_n^i,
\end{equation}
and
\begin{equation}
\label{eq:nth-PS-meaning}
\mathcal P_n=\sum_{\substack{2\leq i\leq C_n\\ i\ \mathrm{prime}}}B_n^i,
\qquad
\mathcal S_n=\sum_{i=0}^{C_n}\pi(i)B_n^i.
\end{equation}
\end{lemma}

\begin{proof}
In $\mathcal R_n G_1(B_n,C_n+1)$, formula \eqref{eq:nth-C-meaning} shows that
the coefficient in the field indexed by $i+(C_n+1)j$ is
\[
d=\frac{iF_n}{j}.
\]
These fields are distinct. Write $\delta=\kappa_n-D_n/F_n$, so
$0\leq\delta<1$. For $0\leq i\leq C_n$ and $2\leq j\leq n$,
\[
\frac{d\kappa_n}{D_n}
=
\frac ij+\frac{iF_n\delta}{jD_n},
\qquad
0\leq\frac{iF_n\delta}{jD_n}<\frac1j.
\]
The fractional part of $i/j$ is an integral multiple of $1/j$, hence
\[
\floor{\frac{d\kappa_n}{D_n}}=\floor{\frac ij}.
\]
Furthermore,
\[
0\leq d\kappa_n<C_n(D_n+F_n)<D_n(C_n+1)\leq D_nH_n=B_n.
\]
Thus there are no carries between base-$B_n$ fields. Since $D_n$ and $H_n$
are powers of two, the mask $D_n(H_n-1)$ sends a field value
$0\leq x<B_n=D_nH_n$ to $D_n\floor{x/D_n}$. The repunit in \eqref{eq:nth-Q}
covers every field, and division by $D_n$ yields \eqref{eq:nth-Q-meaning}.

Reducing \eqref{eq:nth-Q-meaning} modulo $A_n-1$ replaces every $A_n^j$ by
$1$. Hence the $i$-th base-$B_n$ digit of $\mathcal W_n$ is
\[
w_i=\sum_{j=2}^{n}\floor{\frac ij}.
\]
The bound $w_i\leq nC_n<C_nF_n<D_n<B_n-1$ prevents carries. For $i\geq 1$,
the difference $w_i-w_{i-1}$ counts divisors of $i$ between $2$ and $n$.
Multiplication by $1-B_n$ and reduction modulo $A_n$ forms these first
differences. The nonnegative multiplier $A_n-B_n+1$ is congruent to $1-B_n$
modulo $A_n$, and subtracting $B_n^2G_0(B_n,n-1)$ removes the self-divisor
from positions $2,\ldots,n$. This proves \eqref{eq:nth-T-meaning}.

Every composite $i\leq C_n<(n+1)^2$ has a prime divisor at most $n$ and
strictly less than $i$. Therefore, for $i\geq 2$, $t_i=0$ exactly when $i$
is prime. Also $0\leq t_i\leq n<H_n$. In a base-$B_n$ field, the bit of
weight $H_n$ in $t_i+H_n-1$ is $1$ exactly when $t_i>0$. Since these values
are below $2H_n<B_n$, the second AND in \eqref{eq:nth-P} produces the
composite flags without cross-field carries. Subtracting those flags,
together with the unwanted positions $0$ and $1$, gives the formula for
$\mathcal P_n$.

Finally, multiplication by the repunit
$\mathcal U_n=1+B_n+\cdots+B_n^{C_n}$ forms prefix sums of the prime flags.
These counts are at most $C_n<H_n<B_n$, so reduction modulo
$A_n=B_n^{C_n+1}$ creates no carry between fields. This proves the formula
for $\mathcal S_n$.
\end{proof}

\begin{corollary}[Extracting $\pi(n)$ from the packed sieve]
\label{cor:pi-extraction}
The packed prime flags and prefix counts of \cref{lem:nth-sieve} give
\begin{equation}
\label{eq:pi-from-P}
\pi(n)
=
\bigl(\mathcal P_n\bmod B_n^{n+1}\bigr)\bmod(B_n-1),
\end{equation}
and equivalently
\begin{equation}
\label{eq:pi-from-S}
\pi(n)
=
\floor{\frac{\mathcal S_n}{B_n^n}}\bmod B_n.
\end{equation}
\end{corollary}

\begin{proof}
By \eqref{eq:nth-PS-meaning}, reduction of $\mathcal P_n$ modulo
$B_n^{n+1}$ leaves exactly
\[
\sum_{\substack{p\leq n\\p\ \mathrm{prime}}}B_n^p.
\]
Since $B_n\equiv1\pmod{B_n-1}$, its remainder modulo $B_n-1$ is the number
of summands, namely $\pi(n)$.  This is the actual remainder because
$0\leq\pi(n)\leq n<B_n-1$.  Formula \eqref{eq:pi-from-S} is simply the
$n$-th base-$B_n$ digit of $\mathcal S_n$.
\end{proof}

The preceding corollary already gives $\pi(n)$ from the construction for
$p_n$, but it carries an unnecessarily long candidate range.  For the
prime-counting function alone, the sieve may stop at $n$.

\begin{theorem}[A dedicated packed program for $\pi(n)$]
\label{thm:pi-packed}
Keep $F_n,D_n,H_n,$ and $B_n$ from
\eqref{eq:nth-parameters-a}--\eqref{eq:nth-parameters-b}, but replace the
candidate cutoff $C_n=n(n+1)$ by
\begin{equation}
\label{eq:pi-cutoff}
C_n^{(\pi)}=n.
\end{equation}
Define
\begin{equation}
\label{eq:pi-packed-parameters}
A_n^{(\pi)}=B_n^{n+1},
\qquad
E_n^{(\pi)}=(A_n^{(\pi)})^{n+1},
\qquad
\mathcal U_n^{(\pi)}=G_0(B_n,n+1),
\end{equation}
and let
\[
\mathcal R_n^{(\pi)},\ \mathcal Q_n^{(\pi)},\
\mathcal W_n^{(\pi)},\ \Delta_n^{(\pi)},\
\mathcal P_n^{(\pi)}
\]
be the quantities defined by
\eqref{eq:nth-C}--\eqref{eq:nth-P} after replacing
$C_n,A_n,E_n,\mathcal U_n$ by
$C_n^{(\pi)},A_n^{(\pi)},E_n^{(\pi)},\mathcal U_n^{(\pi)}$ throughout.
Then
\begin{equation}
\label{eq:pi-packed-answer}
\widehat\pi(n):=\mathcal P_n^{(\pi)}\bmod(B_n-1)=\pi(n).
\end{equation}
The construction stops before the prefix-sum and carry-selection stages and
uses no binomial valuation.
\end{theorem}

\begin{proof}
All bounds used in \cref{lem:nth-sieve} remain valid after replacing
$C_n$ by the smaller cutoff $n$; in particular
$nF_n<C_nF_n<D_n$.  The reciprocal-packing argument and the quotient-field
argument are therefore unchanged.  The first-difference step gives
proper-divisor counts for every $0\leq i\leq n$.  Every composite
$i\leq n$ has a proper divisor between $2$ and $n$, so the same binary mask
produces
\[
\mathcal P_n^{(\pi)}
=
\sum_{\substack{2\leq i\leq n\\i\ \mathrm{prime}}}B_n^i.
\]
Reducing modulo $B_n-1$ replaces each power $B_n^i$ by $1$.  The result is
$\pi(n)$, and it lies below $B_n-1$, so the least nonnegative remainder is
exactly $\pi(n)$.
\end{proof}

\begin{lemma}[Rank selection by carries]
\label{lem:nth-carries}
Let $B,H$ be powers of two with $B>4H$, let
\[
S_0=\sum_{i=0}^{K}c_iB^i,
\qquad 0\leq c_i<H,
\]
and put $R=G_0(B,K+1)$. If $1\leq r<H$, then
\begin{equation}
\label{eq:nth-carry-selector}
\nu_2\!\binom{S_0+(2H-r)R}{HR}
=
\#\{i\in\{0,\ldots,K\}:c_i\geq r\}.
\end{equation}
\end{lemma}

\begin{proof}
For nonnegative integers $u,v$, Kummer's theorem \cite{Kummer} identifies
\[
\nu_2\binom{u+v}{u}
\]
with the number of binary carries in the addition of $u$ and $v$.
Equivalently, this follows from Legendre's identity
$\nu_2(k!)=k-\HW(k)$, where $\HW(k)$ is the number of ones in the binary
expansion of $k$.

Take
\[
u=HR,
\qquad
v=S_0+(H-r)R.
\]
In the $i$-th base-$B$ field one adds $H$ to $H-r+c_i$. If $c_i<r$, the
latter is below $H$ and there is no carry. If $c_i\geq r$, it is $H+d$
with $0\leq d<H$, and exactly one carry occurs from the $H$ bit to the $2H$
bit. Since every field remains below $3H<B$, carries never pass between
base-$B$ fields. This proves \eqref{eq:nth-carry-selector}.
\end{proof}

\begin{proof}[Proof of \cref{thm:nth-packed}]
Apply \cref{lem:nth-carries} with
\[
K=C_n,\qquad B=B_n,\qquad H=H_n,\qquad R=\mathcal U_n,
\qquad c_i=\pi(i),\qquad r=n.
\]
Exactly the positions
\[
i=p_n,p_n+1,\ldots,C_n
\]
have $\pi(i)\geq n$, so their number is $C_n-p_n+1$. Therefore
\[
\nu_2\!\binom{
\mathcal S_n+(2H_n-n)\mathcal U_n
}{H_n\mathcal U_n}
=C_n-p_n+1,
\]
and \eqref{eq:nth-answer} gives $\widehat{p}_n=p_n$. The preceding lemmas
also establish every exact division and nonnegativity assertion used in the
construction.
\end{proof}

The same packed prime flags also give an exact count of twin-prime
pairs. We use the convention
\[
\pi_2(n)=\#\{p\geq2:p\text{ and }p+2\text{ are prime},\ p+2\leq n\},
\]
so both members of each pair must lie at most $n$.
The construction below takes only $n$ as input and uses the explicit
formulas of \cref{thm:pi-packed}.

\begin{theorem}[A fixed program for twin-prime counting]
\label{thm:twin-packed}
For $n\geq1$, let $\mathcal P_n^{(\pi)}$ be the integer defined by
the fixed formulas in \cref{thm:pi-packed}, and retain
\[
B_n=D_nH_n=2^{(n+2)^2+n+1}.
\]
Define
\begin{equation}
\label{eq:twin-packed-answer}
\widehat\pi_2(n)
=
\left(
\mathcal P_n^{(\pi)}
\bitand
\floor{\frac{\mathcal P_n^{(\pi)}}{B_n^2}}
\right)\bmod(B_n-1).
\end{equation}
Then
\[
\widehat\pi_2(n)=\pi_2(n).
\]
After substituting the arithmetic-term representations of the
auxiliary operations, \eqref{eq:twin-packed-answer} is a fixed
arithmetic term in $n$. Its defining expression requires no bounded
summation or product operator, recurrence, search, or primality test.
The packed construction uses three bitwise AND operations in total
and requires neither prefix sums nor a binomial carry selector.
\end{theorem}

\begin{proof}
We first prove the counting identity and then verify the
arithmetic-term assertion.

By \cref{thm:pi-packed}, the explicitly constructed integer satisfies
\begin{equation}
\label{eq:twin-packed-prime-flags}
\mathcal P_n^{(\pi)}
=\sum_{\substack{2\leq p\leq n\\p\ \mathrm{prime}}}B_n^p.
\end{equation}
This identity describes the digits of the integer already defined
by the packed formulas; the sum is used only in the proof.

Since $B_n$ is a power of $2$, each nonzero base-$B_n$ digit in
\eqref{eq:twin-packed-prime-flags} is a single set bit. Division
by $B_n^2$ shifts every such bit by two base-$B_n$ positions.
In fact, the division is exact, because every exponent in
\eqref{eq:twin-packed-prime-flags} is at least $2$; this also
holds when the integer is $0$. Consequently,
\[
\floor{\frac{\mathcal P_n^{(\pi)}}{B_n^2}}
=\sum_{\substack{2\leq p\leq n\\p\ \mathrm{prime}}}B_n^{p-2}.
\]
The bit in position corresponding to $B_n^m$ occurs in both
integers precisely when $m$ and $m+2$ are prime and $m+2\leq n$.
It follows that
\begin{equation}
\label{eq:twin-packed-pair-flags}
\mathcal P_n^{(\pi)}
\bitand
\floor{\frac{\mathcal P_n^{(\pi)}}{B_n^2}}
=\sum_{\substack{p,\,p+2\ \mathrm{prime}\\p+2\leq n}}B_n^p.
\end{equation}
Modulo $B_n-1$, every summand on the right is congruent to $1$.
Their number is $\pi_2(n)$, and
\[
0\leq\pi_2(n)\leq n<B_n-1.
\]
Thus the least nonnegative remainder is exactly $\pi_2(n)$,
proving \eqref{eq:twin-packed-answer}, including the cases
$1\leq n\leq4$, when the count is $0$.

To verify that the construction is a fixed arithmetic term,
we start with its defining formulas, rather than the digit
identities \eqref{eq:twin-packed-prime-flags} and
\eqref{eq:twin-packed-pair-flags}. The integer
$\mathcal P_n^{(\pi)}$ is defined by the fixed list of expressions
\eqref{eq:nth-C}--\eqref{eq:nth-P}, with the substitutions in
\cref{thm:pi-packed}. The number of these expressions is
independent of $n$.

The geometric expressions $G_0(x,m)$ and $G_1(x,m)$ have the closed forms in \eqref{eq:twin-packed-geometric-terms}.

Replace the factorial $(n+1)!$ by its fixed arithmetic-term
representation cited in Section~\ref{sec:arith-context}. Replace
each bitwise AND by \cref{thm:and-term}, and replace the
remaining auxiliary operations in that representation by the
terms cited in the same section. Remainders are eliminated
using
\[
a\bmod b=a-b\floor{\frac ab},\qquad a\geq0,\quad b\geq1.
\]
All other exact quotients can likewise be written as floor
divisions. The remaining expressions use arithmetic operations,
exponentiation, and floor division.

Expanding the auxiliary calls first gives a straight-line program in the
six-operation basis. Successively substituting all register definitions
then produces a single expression of fixed syntactic length in $n$.
The parameter $n$ changes the values of its subexpressions and
exponents, but not the number of operations in its syntax.
In particular, none of the prime-indexed sums displayed in this
proof is an operation in that expression. Finally, the
construction of $\mathcal P_n^{(\pi)}$ uses two AND operations,
and \eqref{eq:twin-packed-answer} adds one, establishing the
stated total.
\end{proof}

For example, $\pi_2(20)=4$, corresponding to
$(3,5),(5,7),(11,13)$, and $(17,19)$.
If instead twin pairs are counted by the bound $p\leq n$ on
their smaller member, the resulting counting function is
$\pi_2(n+2)$ and is obtained by replacing $n$ by $n+2$ in
\eqref{eq:twin-packed-answer}.

\subsection{Arithmetic-term closure and integer sizes}
\label{sec:nth-strict}

\cref{thm:nth-packed} is intentionally written with standard operations.
Turning it into a strict arithmetic term requires only substitution of
established results.
The only auxiliary operation specific to the present construction is bitwise
AND. \cref{thm:and-term}, proved in \cref{app:additional-terms}, supplies
an arithmetic representation of it. Another arithmetic term proved in the appendix produces $\lfloor \log_2 n \rfloor$.

\begin{corollary}[Strict arithmetic terms for $p_n$ and $\pi(n)$]
\label{cor:nth-strict}
For every $n\geq1$, both $\widehat{p}_n$ in \eqref{eq:nth-answer} and
$\widehat\pi(n)$ in \eqref{eq:pi-packed-answer} can be expanded into fixed
arithmetic terms in $n$.  Replace every bitwise $\AND$ by
\cref{thm:and-term} and replace the remaining standard auxiliary operations
by the cited arithmetic-term representations.  The resulting terms have
fixed syntactic length and return $p_n$ and $\pi(n)$, respectively.
\end{corollary}

\begin{proof}
\cref{thm:nth-packed,thm:pi-packed} prove the two packed identities, and
\cref{thm:and-term} proves the required identity for every binary mask.  The
cited results prove that factorial, binomial coefficient, $\gcd$, $2$-adic
valuation, and integer logarithm admit arithmetic-term representations on
the arguments used here.  Substitution preserves every intermediate value.
The dedicated $\pi(n)$ term is shorter: it stops at the packed prime flags
and does not use the prefix-sum or binomial carry selector.
\end{proof}

\begin{remark}[Size]
\label{rem:nth-size}
The arithmetic operation counts are not bit-complexity bounds. The
largest modulus in the reciprocal-packing step is $E_n^2$, whose bit length
is
\[
2(h_n+n+1)(C_n+1)(n+1)+1.
\]
With $h_n=(n+2)^2$ this is $O(n^5)$. If one permits the more compact
evaluation choice
\[
h_n=1+\floor{\log_2(C_nF_n)},
\]
which still gives $2^{h_n}>C_nF_n$, the same packed construction has an
$E_n^2$ modulus of $O(n^4\log n)$ bits. These bounds concern an optimized
evaluation of the packed construction, not literal materialization of the
cited arithmetic-term expansions of the standard operations or of
\cref{thm:and-term}.

For the dedicated $\pi(n)$ construction of \cref{thm:pi-packed}, the
largest reciprocal-packing modulus is $(E_n^{(\pi)})^2$, with bit length
\[
2(h_n+n+1)(n+1)^2+1.
\]
Thus the fully arithmetic choice $h_n=(n+2)^2$ gives $O(n^4)$ bits.  With
the compact choice
\[
h_n=1+\floor{\log_2(nF_n)},
\]
the corresponding bound is $O(n^3\log n)$ bits.  This reduction reflects
the fact that prime counting needs candidates only through $n$ and requires
no rank-selection stage.

Literal substitution of \cref{thm:and-term} is far larger still. Its
exponent $m=2^a$ gives the modulus $X^m-2$ a bit length of order $\ell4^a$,
which is exponential in the value of the argument $a$ and hence doubly
exponential in its bit length. This applies to both AND operations in the
packed constructions and to the update $U-\AND(U,V)$ of the recurrence.
Since AND is symmetric, one may take $a$ to be the smaller argument, but the
growth remains of this type. The program length is unaffected. The log-free AND program in
Appendix~\ref{sec:slp-aux} uses a larger admissible parameter to reduce
its operation count; it also has exponential intermediate bit lengths
in the numerical values of its inputs.
\end{remark}

Prunescu and Shunia obtain $p_n$ and $\pi(n)$ by a hypercube encoding of a
number-theoretic predicate \cite{PrunescuShuniaPrime}. They ask whether a
simpler arithmetic term for $p_n$ exists, and whether both functions can be
constructed without the hypercube method.  The present sieve is packed
directly: reciprocal fields produce integer quotients, first differences
count proper divisors, and a binary mask creates prime flags.  For $\pi(n)$
the construction can stop there; for $p_n$, prefix sums form the values
$\pi(i)$ and one binomial valuation selects the first index of rank $n$.
The terms are still enormous, but every packed object is visible in the
proof and the construction is not a hypercube count.

\section{Conclusion}
\label{sec:closing}

The autonomous state $T_k$ stores both a reduced-residue pattern and the
control data for its own update. Its binary expansion therefore carries
the sieve together with the information needed to advance it. Its valuation is the next odd prime, and its first
two occupied positions determine the following prime gap. The related
state $Z_x$ synchronizes wheel updates with the ordinary integers and
satisfies $\nu_2(Z_x)=\pi(x)$.

The direct packed sieve starts from $n$. It constructs a list
of proper-divisor counts, converts it into prime flags, and either sums
those flags or selects the first index of rank $n$ by counting carries.
Shifting and intersecting the same flags counts twin-prime pairs.
The straight-line formulations make the size of these descriptions
precise: their operation counts include arithmetic implementations of
all auxiliary calls and allow computed registers to be reused.

Unfolding these programs recovers fixed arithmetic terms. Their short
instruction lists do not make them efficient prime algorithms: the
intermediate integers can be enormous, while ordinary sieving remains
far more practical \cite[\S1.5]{HardyWright}. The distinctions between
program length, expanded term size, and intermediate bit length are
therefore essential. Further questions concern reducing the number of
instructions, the packed candidate range, and the integer sizes required
by the arithmetic implementation of AND.

\appendix
\section{Arithmetic subroutines and their proofs}
\label{app:additional-terms}

This appendix collects arithmetic-term representations of auxiliary
operations.  The first, an arithmetic term for bitwise AND, is used in the
main text: it removes bitwise AND as a primitive from both recurrences and
from the packed terms for $p_n$ and $\pi(n)$.  The remaining two results show
that the binary logarithm $\floor{\log_2}$ also has an elementary
arithmetic-term construction: an arithmetic term for integer roots yields
one for $\floor{\log_2}$.  This operation occurs in the recurrence, through
$L=\floor{\log_2T_k}-1$, and in the AND term, through the parameter $r$.
The integer-root term also proves a formula conjectured in
\cite{ShuniaKronecker}.  Both results use the same general idea as the AND
term: an analytic or algebraic operation is replaced by a fixed modular
expression.

\subsection{Reusable arithmetic subroutines}
\label{sec:slp-aux}

We give schedules for the auxiliary operations used in our programs.
The assignments in each display are evaluated in their written order,
row by row. Parentheses determine the operations within each assignment,
and a computed register may be reused at no cost.
The six charged operations are addition, truncated subtraction,
multiplication, floor division, exponentiation, and remainder. In this
subsection the notation $\mathsf{B}_+(N,j)$ is restricted to $N\geq1$,
whereas $\mathsf{B}(N,j)$ also allows $N=0$.

\paragraph{Gcd: $13$ operations.}
For $a,b\geq1$, compute
\begin{align*}
 c&=ab,& B&=5^c,& s&=c+a+b,\\
 P&=B^s,& Q&=(B^a\mathbin{\dot{-}}1)
                     (B^b\mathbin{\dot{-}}1),\\
 \mathsf{D}(a,b)&=
 \left(\left\lfloor P/Q\right\rfloor\bmod B\right)
       \mathbin{\dot{-}}1.
\end{align*}
The costs are $1+1+2+1+5+3=13$.
This is the gcd formula of \cite{PrunescuShuniaGCD}, with the common
power $5^{ab}$ computed once. Its output is $\gcd(a,b)$.
Both factors in $Q$ are positive.

\paragraph{Valuation: $19$ operations.}
For $n\geq1$, compute
\[
 B=2^n,\qquad g=\mathsf{D}(n,B),\qquad u=B\mathbin{\dot{-}}1,
 \qquad
 \mathsf{V}(n)=\left\lfloor(g^n\bmod u^2)/u\right\rfloor.
\]
The cost is $1+13+1+4=19$.
To verify this valuation formula, write $v=\nu_2(n)$, so that $g=2^v$ and
$g^n=(u+1)^v$. For $n\geq2$, one has $v<n$ and
$1+vu<u^2$, hence the remainder modulo $u^2$ is $1+vu$.
The final quotient is $v$. For $n=1$, the displayed program gives
$0$ directly. The register $g=2^{\mathsf{V}(n)}$ is available if a
caller also needs it.

\paragraph{Binomial coefficients: $6$ or $7$ operations.}
For $N\geq1$ and $0\leq j\leq N$, put
\[
 B=2^N,\qquad P=(B+1)^N,\qquad Q=B^j,\qquad
 \mathsf{B}_+(N,j)=\left\lfloor P/Q\right\rfloor\bmod B.
\]
The cost is $1+2+1+2=6$. Every coefficient in $(1+X)^N$
is smaller than $2^N$, so extraction of its $j$-th base-$B$ digit
returns $\binom Nj$. For a version $\mathsf{B}(N,j)$ defined also
when $N=j=0$, replace the first assignment by $B=2^{N+1}$.
This costs $7$ operations and gives the same binomial coefficient
throughout $0\leq j\leq N$.

\paragraph{Factorial: $10$ operations.}
For $n\geq0$, compute
\[
 q=n^2,\qquad b=8^q,\qquad z=b^n,\qquad
 C=\mathsf{B}_+(b,n),\qquad
 \mathsf{F}(n)=\left\lfloor z/C\right\rfloor.
\]
The cost is $1+1+1+6+1=10$.
The identity $\mathsf{F}(n)=n!$ is the factorial formula of
\cite{PrunescuSaurasFactorial}. Its binomial call is valid because
$b\geq1$ and $n\leq b$, including $n=0$.

\paragraph{Hamming weight: $26$ operations, including zero.}
For $n\geq0$, compute
\begin{align*}
 t&=2n,& e&=t+1,& B&=2^e,\\
 C&=\left\lfloor (B+1)^t/B^n\right\rfloor\bmod B,\\
 g&=\mathsf{D}(C,B),& u&=B\mathbin{\dot{-}}1,\\
 \mathsf{H}(n)&=\left\lfloor(g^e\bmod u^2)/u\right\rfloor.
\end{align*}
The first two lines cost $8$ operations; the gcd costs $13$;
and the final two assignments cost $5$. Thus the total is $26$.
The binomial extraction gives $C=\binom{2n}{n}$, including $C=1$
at $n=0$. Kummer's identity gives
$\nu_2(C)=\HW(n)<e$, so $g=2^{\HW(n)}$. The same modular
extraction as above returns $\HW(n)$ for $n\geq1$. At $n=0$
the program gives $u=1$ and output $0$. In particular, no branch
or undefined valuation at zero occurs.

\paragraph{Bitwise AND: $34$ operations.}
The proof of \cref{thm:and-term} only requires that
$\ell=2^r-1\geq a+b$ and $r\geq1$. Consequently we may take
$r=a+b+1$, replacing the logarithm by a larger arithmetic width.
For $a,b\geq0$, the program $\mathsf{A}(a,b)$ is
\begin{align*}
 r&=a+b+1,& \ell&=2^r\mathbin{\dot{-}}1,& m&=2^a,\\
 X&=2^{\ell m+a+2},&Y&=(X+1)^\ell+1,\\
 P&=(1+XY)^a,& Q&=(1+Y)^{\ell\mathbin{\dot{-}}a},& W&=Y^b,\\
 J&=\left\lfloor PQ/W\right\rfloor\bmod Y,\\
 E&=X^m\mathbin{\dot{-}}2,& M&=2m,&
 R&=(J^m\bmod E)\bmod M,\\
 u&=M\mathbin{\dot{-}}1,& e&=a+1,&
 \mathsf{A}(a,b)&=\left\lfloor(R^e\bmod u^2)/u\right\rfloor.
\end{align*}
The successive row costs are $5,7,7,3,6,6$, for a total of
$34$. The first five rows are the construction of
\cref{thm:and-term} with the permitted larger value of $r$.
Its proof shows that $R=2^d$, where $d=a\bitand b$ and
$0\leq d\leq a$. Since $M=2^{a+1}$, the last row extracts
$d$ from $R^{a+1}=(u+1)^d$. If $a\geq1$, then
$1+du<u^2$ and the extraction is exact; if $a=0$, then
$d=0$ and $u=1$, so the output is again correct. Therefore
$\mathsf{A}(a,b)=a\bitand b$ for every $a,b\geq0$.

This schedule favors operation count over intermediate size.
The logarithmic choice of $r$ in \cref{thm:and-term} gives
smaller intermediate integers. We retain that theorem and its proof
because it describes this alternative without altering the packed-sieve
or recurrence arguments.

\paragraph{Integer roots: $12$ operations.}
Under the hypotheses of \cref{thm:integer-root}, compute
\begin{align*}
 h&=2m,& s&=hn,& b&=n^s,& D&=b^m\mathbin{\dot{-}}n,\\
 x&=b+1,& Z&=x^s,& P&=(Zx)\bmod D,& Q&=Z\bmod D,\\
 \mathsf{R}(n,m)&=\left\lfloor P/Q\right\rfloor
                    \mathbin{\dot{-}}1.
\end{align*}
The costs are $5+5+2=12$.
Here $Zx=(b+1)^{s+1}$, so the output is exactly the expression
proved in \cref{thm:integer-root}. In particular $Q>0$.
We write $\mathsf{R}_h(n,m,h)$ for the same schedule with its
first assignment omitted when $h=2m$ is supplied; it costs $11$.
Two root calls with the same index therefore cost $23$ operations
when the register $h$ is shared.

\paragraph{Binary logarithm: $35$ operations for $n\geq2$.}
The program implementing \cref{thm:binary-log-root} is
\begin{align*}
 k&=3n^2,& Q&=(16k^2)(n+1),& W&=Q^k,\\
 N_1&=nW,&N_2&=2W,&h&=2k,\\
 u&=\mathsf{R}_h(N_1,k,h),&v&=\mathsf{R}_h(N_2,k,h),\\
 \mathsf{L}_2(n)&=\left\lfloor
 (u\mathbin{\dot{-}}Q)/(v\mathbin{\dot{-}}Q)
 \right\rfloor.
\end{align*}
The costs are $7+3+22+3=35$. The cited theorem establishes
that both root calls satisfy their domain conditions, that the final
divisor is positive, and that the result is $\lfloor\log_2n\rfloor$.
A $38$-operation extension to every $n\geq1$ is
\[
 z=\left\lfloor1/n\right\rfloor,\qquad
 \mathsf{L}(n)=\mathsf{L}_2(n+z)\mathbin{\dot{-}}z.
\]
Indeed $z=1$ only at $n=1$; hence the argument of
$\mathsf{L}_2$ is always at least $2$, and the output at $n=1$ is $0$.

\subsection{An arithmetic term for bitwise AND}
\label{sec:and-term}

The binary bitwise conjunction $a \, \& \, b$ is a convenient notation for the clearing step in
\eqref{eq:parameters_integer} and in the packed sieve, but it is not an essential primitive. Since bitwise AND is Kalm\'ar elementary, Mazzanti's theorem~\cite{Mazzanti} guarantees that some arithmetic term represents it. Here we construct an explicit and comparatively short one. We use $x\bmod y$ for the least nonnegative remainder of $x$ upon
division by the positive integer $y$. We observe that $x \bmod y = x - y \cdot \lfloor x / y \rfloor$ is an arithmetic term.

\begin{theorem}[An arithmetic term for binary bitwise conjunction]
\label{thm:and-term}
For integers $a,b\geq 0$, define
\begin{equation}
\label{eq:and-parameters}
r=1+\floor{\log_2(a+b+1)},\qquad
\ell=2^r-1,\qquad
m=2^a,
\end{equation}
\begin{equation}
\label{eq:and-XY}
X=2^{\ell m+a+2},\qquad
Y=(X+1)^\ell+1,
\end{equation}
and
\begin{equation}
\label{eq:and-J}
J=
\left(
\floor{\frac{(1+XY)^a(1+Y)^{\ell-a}}{Y^b}}
\bmod Y
\right).
\end{equation}
Then
\begin{equation}
\label{eq:and-term}
\AND(a,b)=
\nu_2\!\left(
\bigl(J^m\bmod(X^m-2)\bigr)\bmod(2m)
\right)
\end{equation}
satisfies
\begin{equation}
\label{eq:and-identity}
\AND(a,b)=a\bitand b.
\end{equation}
\end{theorem}

\begin{proof}
By construction $\ell=2^r-1\geq a+b$, so both $a$ and $b$ fit in $r$ binary
digits. Here we observe that $r = a + b + 1$ would also work, with less arithmetic operations but more digits to handle. Now consider
\[
F(Z,T)=(1+ZT)^a(1+T)^{\ell-a}
      =\sum_{j=0}^{\ell}F_j(Z)T^j.
\]
All coefficients of the $F_j$ are nonnegative, and
\[
F_j(X)\leq F(X,1)=(1+X)^a2^{\ell-a}
\leq(1+X)^\ell<Y.
\]
Thus the numbers $F_j(X)$ are valid base-$Y$ digits. Dividing $F(X,Y)$ by
$Y^b$, taking the floor, and reducing modulo $Y$ extracts the $b$-th digit:
\begin{equation}
\label{eq:J-coefficient}
J=F_b(X).
\end{equation}

Write
\[
a=\sum_{j=0}^{r-1}a_j2^j,
\qquad
b=\sum_{j=0}^{r-1}b_j2^j,
\qquad a_j,b_j\in\{0,1\}.
\]
Because $\ell$ has $r$ binary digits all equal to $1$, the integer $\ell-a$
is the $r$-bit complement of $a$. Over $\mathbb{F}_2[Z,T]$, Frobenius gives
\[
F(Z,T)\equiv
\prod_{a_j=1}\left(1+Z^{2^j}T^{2^j}\right)
\prod_{a_j=0}\left(1+T^{2^j}\right)
\pmod2.
\]
To obtain the coefficient of $T^b$, each factor is forced by the bit $b_j$.
Hence
\begin{equation}
\label{eq:Fb-parity}
F_b(Z)\equiv Z^d\pmod2,
\qquad
d=\sum_{j=0}^{r-1}a_jb_j2^j=a\bitand b.
\end{equation}
Since the coefficients of $F_b$ are nonnegative, there exists
$H\in\Z_{\geq 0}[Z]$ such that
\begin{equation}
\label{eq:Fb-G}
F_b(Z)=Z^d+2H(Z).
\end{equation}

Now $m=2^a$. For $1\leq s<m$,
\[
\nu_2\!\binom{m}{s}=a-\nu_2(s),
\]
because
\(
\binom{m}{s}=(m/s)\binom{m-1}{s-1}
\)
and every binomial coefficient in row $m-1=2^a-1$ is odd. Therefore
\[
\nu_2\!\left(2^s\binom{m}{s}\right)
=s+a-\nu_2(s)\geq a+1.
\]
The final term $(2H)^m$ is also divisible by $2^{a+1}=2m$, since
$m=2^a\geq a+1$. Expanding \eqref{eq:Fb-G} to the $m$-th power gives
\begin{equation}
\label{eq:Fbm}
F_b(Z)^m=Z^{dm}+2mK(Z)
\end{equation}
for some $K\in\Z_{\geq 0}[Z]$.

Reduce this polynomial modulo $Z^m-2$. Each monomial $Z^e$ becomes
\[
2^{\floor{e/m}}Z^{\,e\bmod m}.
\]
Hence the polynomial remainder has the form
\begin{equation}
\label{eq:poly-rem}
R(Z)=2^d+2mK_0(Z),\qquad \deg R<m,
\end{equation}
with nonnegative coefficients. We next verify that evaluating this
polynomial remainder at $X$ also gives the ordinary integer remainder in
\eqref{eq:and-term}. By Vandermonde's identity,
\[
F_b(1)=\binom{\ell}{b}\leq 2^\ell,
\]
so the sum of the coefficients of $F_b(Z)^m$ is at most $2^{\ell m}$. Since
$\deg F_b^m\leq am$, reduction modulo $Z^m-2$ multiplies the coefficient of
any monomial by at most $2^a$. Thus the sum of the coefficients of $R$ is
at most
\[
2^{\ell m+a}=\frac X4.
\]
Because $\deg R<m$ and $X\geq 8$,
\[
0\leq R(X)\leq \frac X4X^{m-1}
=\frac{X^m}{4}<X^m-2.
\]
It follows that
\[
J^m\bmod(X^m-2)=R(X).
\]
Finally, \eqref{eq:poly-rem} and $0\leq d\leq a$ give
\[
\bigl(J^m\bmod(X^m-2)\bigr)\bmod(2m)
=R(X)\bmod(2m)=2^d,
\]
since $2^d<2^{a+1}=2m$. Taking the $2$-adic valuation yields
$\AND(a,b)=d=a\bitand b$.
\end{proof}

\subsection{An arithmetic term for integer roots}
\label{sec:integer-roots}

The following theorem proves the conjectural integer-root formula introduced
in \cite{ShuniaKronecker}.

\begin{theorem}[Integer-root term]
\label{thm:integer-root}
Let $n,m$ be positive integers satisfying
\[
n>2,
\qquad
2\leq m\leq\floor{\log_2 n}+1,
\]
and suppose that $n$ is not a perfect $m$-th power.  Put
\[
s=2nm,
\qquad
b=n^s,
\qquad
D=b^m-n.
\]
Then
\begin{equation}
\label{eq:integer-root-term}
\floor{\sqrt[m]{n}}
=
\floor{
\frac{(b+1)^{s+1}\bmod D}{(b+1)^s\bmod D}
}-1.
\end{equation}
Consequently the integer-root function on this domain is represented by an
arithmetic term.
\end{theorem}

\begin{proof}
Let
\[
\rho=n^{1/m}
\]
and, in the quotient ring $\Z[x]/(x^m-n)$, define
\[
F_t(x)=(x+1)^t\bmod(x^m-n)
=
\sum_{j=0}^{m-1}c_j(t)x^j.
\]
Reduction of the binomial expansion gives
\begin{equation}
\label{eq:root-coefficients-app}
c_j(t)
=
\sum_{\substack{\ell\geq0\\j+\ell m\leq t}}
\binom{t}{j+\ell m}n^\ell.
\end{equation}
All coefficients are nonnegative.  Put
\[
S_t=\sum_{j=0}^{m-1}c_j(t).
\]
Since $n=\rho^m$ and $\rho\geq1$,
\[
S_t
\leq
\sum_{a=0}^{t}\binom ta\rho^a
=(1+\rho)^t.
\]
For $t\leq s+1$, the inequalities $\rho\leq\sqrt n$ and $s\geq4n$ give
\[
S_t\leq(1+\sqrt n)^{s+1}<n^s=b.
\]
Indeed, $(1+\sqrt n)/n<1$ for $n\geq3$, so it is enough to prove
\[
(1+\sqrt n)^{4n+1}<n^{4n}.
\]
Put $x=\sqrt n$ and
\[
g(x)=8x^2\log x-(4x^2+1)\log(1+x).
\]
Using $\log(1+1/x)\leq1/x$,
\[
\begin{aligned}
g'(x)
&=8x\log x-8x\log(1+1/x)+8x-\frac{4x^2+1}{x+1}\\
&\geq4x-4-\frac5{x+1}>0
\qquad(x\geq\sqrt3).
\end{aligned}
\]
Moreover, $\sqrt3<7/4$ and
$11^{13}<3^{12}4^{13}$ imply $g(\sqrt3)>0$.  Hence $g(x)>0$ for
$x\geq\sqrt3$, proving the required estimate.

Thus every $c_j(t)<b$, and more precisely
\[
0<F_t(b)
\leq S_tb^{m-1}
\leq(b-1)b^{m-1}
<b^m-n=D,
\]
where the last inequality follows from $b^{m-1}>n$.  Since
$b^m\equiv n\pmod D$,
\begin{equation}
\label{eq:root-encoding-app}
(b+1)^t\bmod D=F_t(b)
\qquad(t=s,s+1).
\end{equation}

It remains to estimate the quotient in \eqref{eq:integer-root-term}.  Put
\[
P_j(t)=\frac{\rho^jc_j(t)}{(1+\rho)^t}.
\]
Let $\zeta=e^{2\pi i/m}$.  The root-of-unity filter gives
\[
P_j(t)
=
\frac1m\sum_{\ell=0}^{m-1}
\zeta^{-j\ell}
\left(
\frac{1+\rho\zeta^\ell}{1+\rho}
\right)^t.
\]
Writing
\[
\eta^2
=
1-
\frac{2\rho(1-\cos(2\pi/m))}{(1+\rho)^2},
\]
the nontrivial terms have modulus at most $\eta^t$.  We claim
\begin{equation}
\label{eq:root-spectral-app}
E:=(m-1)\eta^s\leq\frac{1}{8mn^2}.
\end{equation}
For $m=2$,
\[
\eta^{4n}
=
\left(\frac{\sqrt n-1}{\sqrt n+1}\right)^{4n}
\leq
\exp\left(-\frac{8n}{\sqrt n+1}\right)
\leq
\frac{1}{16n^2}.
\]
The first inequality follows from
\[
\log\left(\frac{\sqrt n-1}{\sqrt n+1}\right)
=
\log\left(1-\frac2{\sqrt n+1}\right)
\leq-\frac2{\sqrt n+1}.
\]
For the second, put $x=\sqrt n$.  The function
\[
\frac{8x^2}{x+1}-\log(16x^4)
\]
is positive at $x=\sqrt3$ and has derivative
\[
\frac{8x(x+2)}{(x+1)^2}-\frac4x>0
\qquad(x\geq\sqrt3).
\]

For $m\geq3$, using
\[
1-\cos(2\pi/m)\geq\frac{27}{2m^2},
\qquad
(1+\rho)^2\leq3\rho^2,
\qquad
n\geq2^{m-1},
\]
we obtain
\[
\eta^s
\leq
\exp\left(-\frac{9n^{1-1/m}}m\right)
\leq
\frac{1}{8m(m-1)n^2}.
\]
To justify the cosine bound, use
$1-\cos(2\pi/m)=2\sin^2(\pi/m)$ and concavity of $\sin$ on
$[0,\pi/2]$:
\[
\frac{\sin(\pi/m)}{\pi/m}
\geq
\frac{\sin(\pi/3)}{\pi/3}
=
\frac{3\sqrt3}{2\pi}.
\]
Also $n\geq2^{m-1}$ implies
$\rho\geq2^{(m-1)/m}>3/2$, hence $(1+\rho)^2\leq3\rho^2$.  Finally,
$\log(1-u)\leq-u$ gives
\[
\eta^s
\leq
\exp\left(
-\frac{s\rho(1-\cos(2\pi/m))}{(1+\rho)^2}
\right).
\]
The remaining scalar inequality is
\[
\frac{9n^{1-1/m}}m
\geq
\log\bigl(8m(m-1)n^2\bigr).
\]
The difference between the two sides is increasing in $n$ for
$n\geq2^{m-1}$, because its derivative has the sign of
\[
9(m-1)n^{1-1/m}-2m^2>0.
\]
At the endpoint it is positive for $m=3,4$ by direct substitution.  For
$m\geq5$, the left side is greater than $9\cdot2^{m-2}/m$, while the right
side is less than $2.6m$; the inequality follows from
$2^{m-2}>(2.6/9)m^2$, first at $m=5$ and then by induction.  This proves
\eqref{eq:root-spectral-app}.

Hence
\begin{equation}
\label{eq:root-uniformity-app}
\left|P_j(t)-\frac1m\right|
\leq
\frac Em
\qquad(t=s,s+1).
\end{equation}
Let
\[
C=c_{m-1}(s),
\qquad
r=\frac{c_{m-2}(s)}{c_{m-1}(s)}.
\]
From \eqref{eq:root-uniformity-app},
\begin{equation}
\label{eq:root-leading-ratio-app}
r
=
\rho\frac{P_{m-2}(s)}{P_{m-1}(s)},
\qquad
|r-\rho|
\leq
\frac{2\rho E}{1-E}
<
\frac{\rho}{3mn^2}.
\end{equation}
Multiplication by $x+1$ in the quotient ring shows that the leading
coefficient of $F_{s+1}$ is $C(1+r)$.

The lower base-$b$ digits cause a smaller error.  Equation
\eqref{eq:root-uniformity-app} implies, for $0\leq j<m-1$ and
$t\in\{s,s+1\}$,
\[
\frac{c_j(t)}{c_{m-1}(t)}
=
\rho^{m-1-j}\frac{P_j(t)}{P_{m-1}(t)}
<2n.
\]
Consequently,
\[
\delta_t
:=
\sum_{j=0}^{m-2}
\frac{c_j(t)}{c_{m-1}(t)b^{m-1-j}}
<
2n\sum_{r=1}^{m-1}b^{-r}
<
\frac{2n}{b-1}
\leq
\frac{4n}{b}.
\]
Thus
\[
F_s(b)=Cb^{m-1}(1+\delta_s),
\]
\[
F_{s+1}(b)=C(1+r)b^{m-1}(1+\delta_{s+1}),
\]
where
\[
0\leq\delta_s,\delta_{s+1}<\frac{4n}{b}.
\]
Moreover, \eqref{eq:root-leading-ratio-app} gives $1+r<3n$.  Therefore
\[
\begin{aligned}
\left|
\frac{F_{s+1}(b)}{F_s(b)}-(1+\rho)
\right|
&\leq
|r-\rho|
+(1+r)\frac{|\delta_{s+1}-\delta_s|}{1+\delta_s}\\
&<
\frac{\rho}{3mn^2}+\frac{24n^2}{b}.
\end{aligned}
\]
Since $s\geq12$, $b=n^s\geq n^{12}$, and $m\leq n$,
\[
\frac{24n^2}{b}<\frac{\rho}{3mn^2}.
\]
It follows that
\begin{equation}
\label{eq:root-ratio-error-app}
\left|
\frac{F_{s+1}(b)}{F_s(b)}-(1+\rho)
\right|
<
\frac{2\rho}{3mn^2}.
\end{equation}

Let $q=\floor\rho$.  Since $n$ is not a perfect $m$-th power,
\[
1\leq n-q^m
=(\rho-q)\sum_{j=0}^{m-1}\rho^{m-1-j}q^j
\leq m\rho^{m-1}(\rho-q),
\]
so $\rho-q\geq\rho/(mn)$.  Similarly,
\[
1\leq(q+1)^m-n
=(q+1-\rho)\sum_{j=0}^{m-1}(q+1)^{m-1-j}\rho^j.
\]
Since $q+1\leq2\rho$ and $2^{m-1}\leq n$, the final sum is at most
$m2^{m-1}\rho^{m-1}\leq mn^2/\rho$.  Therefore
\begin{equation}
\label{eq:root-gap-app}
\min(\rho-q,q+1-\rho)
\geq
\frac{\rho}{mn^2}.
\end{equation}
Combining \eqref{eq:root-encoding-app}, \eqref{eq:root-ratio-error-app}, and
\eqref{eq:root-gap-app} gives
\[
q
<
\frac{(b+1)^{s+1}\bmod D}{(b+1)^s\bmod D}-1
<
q+1.
\]
Hence
\[
\floor{
\frac{(b+1)^{s+1}\bmod D}{(b+1)^s\bmod D}
}
=q+1,
\]
and \eqref{eq:integer-root-term} follows.
\end{proof}

\subsection{A root-derived arithmetic term for the binary logarithm}
\label{sec:binary-log-root}

The elementary limit
\[
\log n=\lim_{k\to\infty}k(n^{1/k}-1)
\]
suggests comparing the same expression for $n$ and for $2$.  A sufficiently
large but finite root index already determines the integer part of the
binary logarithm.  The integer-root term above then removes the real roots.
The resulting term offers an alternative to the cited representation of
$\floor{\log_2}$. It covers $L=\floor{\log_2T_k}-1$ in the recurrence,
since $T_k\geq8$, and the parameter $r$ in \eqref{eq:and-parameters}
whenever $a+b\geq1$.

\begin{theorem}[Binary logarithm from integer roots]
\label{thm:binary-log-root}
Let $n\geq2$ and put
\[
k=3n^2,
\qquad
Q=16k^2(n+1).
\]
Define the two integer roots
\[
u=\floor{\sqrt[k]{nQ^k}},
\qquad
v=\floor{\sqrt[k]{2Q^k}}.
\]
Then
\begin{equation}
\label{eq:binary-log-root}
\floor{\log_2 n}
=
\floor{\frac{u-Q}{v-Q}}.
\end{equation}
Moreover, both $u$ and $v$ satisfy the hypotheses of
\cref{thm:integer-root}; substituting \eqref{eq:integer-root-term} for these
two roots therefore gives a fixed arithmetic term for
$\floor{\log_2 n}$.
\end{theorem}

\begin{proof}
Write
\[
a=\floor{\log_2 n},
\qquad
x=\log_2 n,
\qquad
\alpha=n^{1/k},
\qquad
\beta=2^{1/k},
\]
and set
\[
R=\frac{\alpha-1}{\beta-1}.
\]
Since $\alpha=\beta^x$ and $t\mapsto\beta^t$ is convex,
\[
R\geq x.
\]
Let $y=(\log n)/k$ and $z=(\log2)/k$.  Using
$e^y-1\leq ye^y$ and $e^z-1\geq z$,
\[
0\leq R-x
\leq
x(e^y-1)
\leq
\frac{(\log n)^2e^y}{k\log2}
<
\frac{2}{3n\log2}.
\]
Here $(\log n)^2<n$ for $n\geq2$, and $y<\log2$.

Because $2^{a+1}$ and $n$ are distinct integers,
$2^{a+1}\geq n+1$, so
\[
a+1-x
=
\frac{\log(2^{a+1}/n)}{\log2}
\geq
\frac{\log(1+1/n)}{\log2}
\geq
\frac{1}{(n+1)\log2}.
\]
For $n\geq3$ it follows that
\[
a+1-R
>
\frac{n-2}{3n(n+1)\log2}
\geq
\frac{\log2}{3n^2}
=
\frac{\log2}{k}.
\]
The last inequality follows from $(\log2)^2<1/2$.

For the lower boundary, if $n$ is not a power of $2$, then
\[
x-a
\geq
\frac{\log(1+1/n)}{\log2}
>
\frac{\log2}{k}.
\]
If $n=2^a$ and $a\geq2$, then
\[
R=1+\beta+\cdots+\beta^{a-1},
\]
so
\[
R-a\geq\beta-1\geq\frac{\log2}{k}.
\]
Thus for every $n\geq3$, the number $R$ lies in the interval
$[a,a+1)$ and is at distance at least
\[
\delta=\frac{\log2}{k}
\]
from either endpoint.  For $n=2$, one has $R=1$ exactly.

Now put
\[
\widetilde\alpha=\frac uQ,
\qquad
\widetilde\beta=\frac vQ,
\qquad
\widetilde R
=
\frac{u-Q}{v-Q}.
\]
The definitions of $u$ and $v$ give
\[
0\leq\alpha-\widetilde\alpha<\frac1Q,
\qquad
0\leq\beta-\widetilde\beta<\frac1Q.
\]
Since $\beta-1\geq(\log2)/k$ and the chosen $Q$ satisfies
$Q(\beta-1)>2$, we have
\[
\widetilde\beta-1\geq\frac{\beta-1}{2}>0.
\]
For $n\geq3$, using $R<a+1$ and $a+2\leq n+1$,
\[
\begin{aligned}
|\widetilde R-R|
&<
\frac{R+1}{Q(\widetilde\beta-1)}\\
&\leq
\frac{2k(n+1)}{Q\log2}
=
\frac{1}{8k\log2}
<
\frac{\log2}{k}
=\delta,
\end{aligned}
\]
where the last inequality uses $\log2>1/2$.  Hence
$\floor{\widetilde R}=a$.  For $n=2$, the two radicands are equal, so
$u=v$ and $\widetilde R=1$.  This proves \eqref{eq:binary-log-root}.

It remains only to justify the arithmetic-term claim.  Since $Q\geq2$,
\[
k\leq\floor{\log_2(nQ^k)}+1,
\qquad
k\leq\floor{\log_2(2Q^k)}+1.
\]
Neither $nQ^k$ nor $2Q^k$ is a perfect $k$-th power.  Indeed, a rational
$k$-th root of the integer factor $n$ would have to be an integer, while
$k>\log_2 n$ rules out $n=d^k$ with $d\geq2$; the same argument applies to
$2$.  Thus \cref{thm:integer-root} applies to both $u$ and $v$, and
substitution into \eqref{eq:binary-log-root} gives the asserted arithmetic
term.
\end{proof}

\section{Straight-line programs and operation counts}
\label{app:programs}

The main text describes the integers produced by the sieve and proves
their invariants. This appendix gives explicit evaluation schedules for
the same constructions. Every count uses the six-operation convention
of Section~\ref{sec:arith-context} and includes the arithmetic subroutines
of Appendix~\ref{sec:slp-aux}. Previously computed registers remain
available for reuse.

\subsection{A straight-line program for one sieve update}
\label{sec:slp-update}

We use the operation-count convention of Section~\ref{sec:arith-context}.
Write $\mathsf V=19$, $\mathsf L=35$, $\mathsf A=34$, and $\mathsf H=26$
for the program lengths for valuation, binary logarithm on inputs at least
$2$, bitwise AND, and Hamming weight given in Appendix~\ref{sec:slp-aux}.
The symbol $a\mathbin{\mathrm{div}}b$ denotes $\floor{a/b}$.
The valuation program already computes $g=2^{\nu_2(t)}$ on input $t$.
We write $\mathsf{V}^*(t)$ for its pair of outputs $(\nu_2(t),g)$;
retaining this register keeps its cost at $19$.

One application of the autonomous update $\Phi$ is computed by the
following fixed program. Its input is a reachable state $T=T_k$; its
output is $r_{15}=T_{k+1}$. All registers remain available after the last
assignment.
\begin{equation}
\label{eq:slp-sieve-update}
\begin{array}{r@{\quad}l@{\qquad}r@{\quad}l}
(r_1,g_T)&\gets\mathsf{V}^*(T)&r_9&\gets2^{r_4}\\
r_2&\gets\lfloor\log_2T\rfloor&r_{10}&\gets r_9\dotminus1\\
r_3&\gets r_2\dotminus1&r_{11}&\gets r_{10}\mathbin{\mathrm{div}}r_6\\
r_4&\gets r_3r_1&r_{12}&\gets Tr_{11}\\
r_5&\gets2^{r_3}&r_{13}&\gets r_{10}\mathbin{\mathrm{div}}r_8\\
r_6&\gets r_5\dotminus1&r_{14}&\gets\AND(r_{12},r_{13})\\
r_7&\gets g_T&r_{15}&\gets r_{12}\dotminus r_{14}\\
r_8&\gets r_7\dotminus1&&
\end{array}
\end{equation}
The left column is executed before the right column. The correspondence
with \eqref{eq:parameters} is
$r_1=\nu_2(T)$, $r_3=L$, $r_{10}=M$, $r_{12}=U$, and $r_{13}=V$.
The common numerator $M$ is computed once and used in both quotients.
The assignment to $r_7$ is an alias of the retained valuation register.
Consequently, inlining the auxiliary programs gives length
\begin{equation}
\label{eq:slp-sieve-cost}
11+\mathsf V+\mathsf L+\mathsf A=99.
\end{equation}
The proof of \cref{thm:main} verifies that every denominator is positive,
that both geometric divisions are exact, and that the final truncated
subtraction is ordinary subtraction on every reachable input.

It is the \emph{single update} $\Phi$ that is a fixed straight-line
program. The unbounded sequence is obtained by iterating that program;
unrolling $k$ updates uses $k(11+\mathsf V+\mathsf L+\mathsf A)$
operations. This statement is distinct from the fixed straight-line
program taking $n$ directly to $p_n$ in \cref{sec:nth-term}.

\paragraph{Further outputs from retained registers.}
The current prime $p_{k+2}=r_1$ costs no additional operation. For
$k\geq1$, append
\[
s_1\gets T\mathbin{\mathrm{div}}r_7,
\qquad s_2\gets s_1\dotminus1,
\qquad s_3\gets\nu_2(s_2).
\]
Then $s_3=p_{k+3}-p_{k+2}$, by \cref{cor:state-gap}; the additional
length is $2+\mathsf V$. A uniform output, valid also at $k=0$, replaces
the final assignment by
\[
s_3\gets s_2+r_5,\qquad s_4\gets\nu_2(s_3),
\]
and has additional length $3+\mathsf V$. Alternatively, if only the
following prime is required and $k\geq1$, the two assignments
\[
s_1\gets T\dotminus r_7,\qquad s_2\gets\nu_2(s_1)
\]
return $p_{k+3}$ with additional length $1+\mathsf V$.

For $k\geq1$, the indicator of $p_{k+3}-p_{k+2}=2$ is obtained by
\[
s_1\gets T\mathbin{\mathrm{div}}r_7,
\qquad s_2\gets s_1\mathbin{\mathrm{div}}4,
\qquad s_3\gets s_2\mathbin{\mathrm{rem}}2.
\]
Indeed, this selects the bit at position $2$ of $T/2^{\nu_2(T)}$;
all positions occupied in $T$ are odd, and
\cref{cor:state-gap} identifies its next occupied position.
The additional length is $3$, or $2$ if $s_1$ has already been computed
for the prime-gap output. Finally, the twin-candidate count is computed
from $T$ by
\[
a_1\gets T\mathbin{\mathrm{div}}4,
\qquad a_2\gets\AND(T,a_1),
\qquad a_3\gets\HW(a_2),
\]
with length $1+\mathsf A+\mathsf H$.
The product in \cref{thm:twin-candidates} describes this output; it is
not an instruction in the program.

\subsection{A straight-line program for the prime-counting update}
\label{sec:slp-pi-update}

The recurrence of \cref{thm:pi-recurrence} has a fixed straight-line
program for its update, with inputs $(x,Z_x)$. The index $x$ is an input
to each update, whereas the stored state is the single integer $Z_x$.
Two simplifications reduce the number of operations while preserving
the recurrence exactly on its invariant states.

First, put $c=\nu_2(Z)$ and $W=(Z/2^c-1)/2$. On the invariant,
$q=\nu_2(W)$ is the least prime greater than $x$, so $q-x\geq1$ and
\begin{equation}
\label{eq:slp-pi-indicator}
\varepsilon=1\mathbin{\mathrm{div}}(q\dotminus x)
\end{equation}
is $1$ exactly when $q=x+1$, and is $0$ otherwise. Thus it equals the
indicator in \eqref{eq:pi-recurrence-indicator}. Second, writing
$P=\Phi(W)$, selection and re-encoding can be combined as
\begin{equation}
\label{eq:slp-pi-fused}
Z^+=Z+\varepsilon\bigl(2^c(4P+2)\dotminus Z\bigr).
\end{equation}
If $\varepsilon=0$, this returns $Z$; if $\varepsilon=1$, it returns
$2^{c+1}(2P+1)$, exactly as in
\eqref{eq:pi-recurrence-state}. Moreover, $P>W$ because the highest
occupied position increases strictly at every wheel update. Hence the
truncated subtraction in \eqref{eq:slp-pi-fused} is ordinary subtraction.

The complete program is
\begin{equation}
\label{eq:slp-pi-update}
\begin{array}{r@{\quad}l@{\qquad}r@{\quad}l}
(z_1,g_Z)&\gets\mathsf{V}^*(Z)&z_9&\gets\Phi(z_5;z_6,g_W)\\
z_2&\gets g_Z&z_{10}&\gets4z_9\\
z_3&\gets Z\mathbin{\mathrm{div}}z_2&z_{11}&\gets z_{10}+2\\
z_4&\gets z_3\dotminus1&z_{12}&\gets z_2z_{11}\\
z_5&\gets z_4\mathbin{\mathrm{div}}2&z_{13}&\gets z_{12}\dotminus Z\\
(z_6,g_W)&\gets\mathsf{V}^*(z_5)&z_{14}&\gets z_8z_{13}\\
z_7&\gets z_6\dotminus x&z_{15}&\gets Z+z_{14}\\
z_8&\gets1\mathbin{\mathrm{div}}z_7&&
\end{array}
\end{equation}
Again the left column precedes the right. Here $\Phi(z_5;z_6,g_W)$ means
execute registers $r_2,\ldots,r_{15}$ of
\eqref{eq:slp-sieve-update} with $T=z_5$ and the already available
$r_1=z_6$ and $g_T=g_W$. Thus the valuation of the decoded wheel is
computed only once, and both powers of two used for decoding and wheel
replication are retained from valuation calls. There are $11$ primitive
operations outside this invocation and $11$ inside it. The total length, after inlining the auxiliary programs,
is therefore
\begin{equation}
\label{eq:slp-pi-cost}
22+2\mathsf V+\mathsf L+\mathsf A=129.
\end{equation}
The output $z_{15}$ is $Z_{x+1}$. If $\pi(x+1)$ is also wanted as a
separate output, append $z_{16}\gets z_1+z_8$, at the cost of one
additional operation. Incrementing the external index $x$ is a separate
addition if it is included in the desired output.

\subsection{Straight-line programs for the packed sieve}
\label{sec:packed-slp}

The packed formulas of Section~\ref{sec:nth-term} can be evaluated with substantial reuse of their
intermediate values.  We give one common program whose cutoff is either
$c=n(n+1)$, for the $n$-th prime, or $c=n$, for the counting functions.
All registers below contain nonnegative integers.  An assignment or an
alias has no cost; each addition, truncated subtraction, multiplication,
floor division, exponentiation, and remainder has cost one.  We use the
straight-line subroutines for factorial, AND, binomial coefficient, and
valuation in Appendix~\ref{sec:slp-aux}, with costs $10$, $34$, $6$, and $19$, respectively.
The six-operation binomial program is used only with a positive top
argument.

Here is the common sieve program.  Each group displays its register
assignments in evaluation order; its cost counts all primitive operations
in the displayed right-hand sides.  Previously computed registers may be
used repeatedly without recomputation.  The symbols $G_0$ and $G_1$ need
not be called as subroutines.
\begin{center}
\small
\begin{tabular}{@{}clc@{}}
\toprule
Block & Assignments & Cost\\
\midrule
1 & $t\gets n+1,\quad u\gets t+1,\quad h\gets u^2$ & \\
  & $D\gets2^h,\quad H\gets2^t,\quad B\gets DH$ & \\
  & $k\gets c+1,\quad A\gets B^k,\quad E\gets A^t$ & $9$\\[1mm]
2 & $F\gets t!$ & $10$\\[1mm]
3 & $b\gets B-1,\quad a\gets A-1$ & \\
  & $U\gets\floor{a/b},\quad G\gets\floor{(E-1)/b}$ & $5$\\[1mm]
4 & $v\gets cB-k,\quad G^{\prime}\gets\floor{(Av+B)/b^2}$ & $6$\\[1mm]
5 & $e\gets E^2,\quad \rho\gets EF,\quad z\gets a^{\rho}\bmod e$ & \\
  & $R\gets\floor{(e+1-z)/E}-FA$ & $9$\\[1mm]
6 & $\kappa\gets\floor{(D+F-1)/F}$ & $3$\\[1mm]
7 & $j\gets H-1,\quad x\gets\kappa R G^{\prime},\quad y\gets DjG$ & \\
  & $Q\gets\floor{\AND(x,y)/D},\quad W\gets Q\bmod a$ & $7+34$\\[1mm]
8 & $v\gets((A-B+1)W)\bmod A,\quad B_2\gets B^2$ & \\
  & $J\gets B^t,\quad d\gets\floor{(J-B_2)/b},\quad\Delta\gets v-d$ & $9$\\[1mm]
9 & $x\gets\Delta+jU,\quad Y\gets HU$ & \\
  & $z\gets\floor{\AND(x,Y)/H},\quad P\gets U-1-B-z$ & $7+34$\\
\bottomrule
\end{tabular}
\end{center}

The ordinary operations in Blocks 1 and 3--9 total
\[
9+5+6+9+3+7+9+7=55.
\]
Thus the common program, with $n$ and $c$ supplied as inputs, has length
$55+10+2\cdot34=133$.  Its output is the packed prime-flag integer $P$;
the registers $U,Y,k,A,b,$ and $B_2$ are retained for subsequent use.

To identify this program with the formulas in Section~\ref{sec:nth-packed}, observe that
\[
U=G_0(B,c+1),\qquad G=G_0(B,(c+1)(n+1)),\qquad
G^{\prime}=G_1(B,c+1).
\]
The third identity follows from
\[
G_1(B,c+1)=\frac{A(cB-(c+1))+B}{(B-1)^2},
\qquad A=B^{c+1}.
\]
In addition, Block 8 reuses the equivalent self-divisor mask
\[
B^2G_0(B,n-1)=\frac{B^{n+1}-B^2}{B-1}.
\]
This quotient is exact and its numerator is nonnegative, including $n=1$.
All remaining identities follow directly from the defining packed
formulas.  Consequently the proofs in Section~\ref{sec:nth-term} establish
correctness and all domains of this straight-line program.

\paragraph{Prime counting.}
For the specialization $c=n$ of \cref{thm:pi-packed}, set $k\gets t$ in Block 1 and
$J\gets A$ in Block 8.  These two aliases replace two arithmetic
instructions, so the packed prime flags cost $131$ operations.  Append
\[
\widehat\pi(n)\gets P\bmod b.
\]
The resulting straight-line program for $\pi(n)$ has length $132$.

\paragraph{Twin-prime counting.}
Let $\pi_2(n)$ count prime pairs $(p,p+2)$ with $p+2\leq n$, as in
\cref{thm:twin-packed}. Use the same $131$-operation flag program and append
\[
x\gets\floor{P/B_2},\qquad y\gets\AND(P,x),\qquad
\widehat\pi_2(n)\gets y\bmod b.
\]
The register $B_2$ has already been computed.  These instructions cost
$1+34+1=36$, giving length $167$ for $\pi_2(n)$.

\paragraph{The $n$-th prime.}
For the specialization $c=n(n+1)$, compute $c\gets nt$ immediately after
$t\gets n+1$.  This adds one multiplication, so the flag program has
length $134$.  Append
\[
\begin{aligned}
S&\gets(PU)\bmod A, & r&\gets2H-n,\\
N&\gets S+rU, & K&\gets\binom{N}{Y},\\
v&\gets\nu_2(K), & \widehat p_n&\gets k-v.
\end{aligned}
\]
Here $N\geq Y\geq1$, so the six-operation binomial subroutine applies
and $K\geq1$ is a valid valuation input.  This final stage costs
\[
2+2+2+6+19+1=32.
\]
The resulting straight-line program for $p_n$ has length $166$.
These are explicit upper bounds furnished by the displayed programs;
no minimality is asserted.  They count arithmetic instructions, not bit
operations, and therefore do not claim practical efficiency.

\section{Practical evaluation code}
\label{app:source}

\subsection{Integer recurrence}
\label{app:python}

The following Python program implements the integer recurrence exactly.
The generator loop iterates a fixed update; it is not included in the
straight-line operation count.
For practical efficiency it uses the machine-level bitwise operation;
\cref{thm:and-term} gives an arithmetic representation if one does not wish
to admit AND as primitive. Here \texttt{T \& -T} isolates the lowest set
bit, so its bit length minus one is $\nu_2(T)$. The value
\texttt{T.bit\_length() - 2} gives $L$. All arithmetic is exact.

\begin{quote}
\begin{verbatim}
def odd_primes():
    T = 8
    while True:
        p = (T & -T).bit_length() - 1
        yield p
        L = T.bit_length() - 2
        M = (1 << (L * p)) - 1
        U = T * (M // ((1 << L) - 1))
        V = M // ((1 << p) - 1)
        T = U - (U & V)

gen = odd_primes()
print([next(gen) for _ in range(7)])
\end{verbatim}
\end{quote}

Running the program prints
\[
[3, 5, 7, 11, 13, 17, 19],
\]
that is $p_2$ through $p_8$. The function yields each prime before
constructing the next, larger state. The prime $2$ may simply be prefixed to
the output.

The following companion function evaluates the one-integer recurrence of
\cref{thm:pi-recurrence}.  It keeps the ordinary index $x$ externally, while
$\nu_2(Z)$ is the corresponding value of $\pi(x)$.

\begin{quote}
\begin{verbatim}
def nu2(t):
    return (t & -t).bit_length() - 1

def wheel_step(T):
    p = nu2(T)
    L = T.bit_length() - 2
    M = (1 << (L*p)) - 1
    U = T * (M // ((1 << L) - 1))
    V = M // ((1 << p) - 1)
    return U - (U & V)

def pi_recurrence(limit):
    Z = 34
    out = []
    for x in range(2, limit + 1):
        out.append((x, nu2(Z)))
        c = nu2(Z)
        W = ((Z >> c) - 1) // 2
        eps = int(x + 1 == nu2(W))
        if eps:
            W = wheel_step(W)
        Z = (1 << (c + eps)) * (2*W + 1)
    return out

print(pi_recurrence(10))
\end{verbatim}
\end{quote}

The second coordinates are
\[
1,2,2,3,3,4,4,4,4,
\]
which are $\pi(2),\ldots,\pi(10)$.

\subsection{Packed program for the \texorpdfstring{$n$-th}{n-th} prime}
\label{app:nth}

The following Python program implements $\widehat{p}_n$ from the formulas in
Section~\ref{sec:nth-term}. As in the recurrence, it uses the machine
bitwise operation rather than \cref{thm:and-term}. The power in
\eqref{eq:nth-C} is reduced modulo $E_n^2$. The binomial coefficient in
\eqref{eq:nth-answer} is not built: its $2$-adic valuation is the
difference of Hamming weights from the proof of \cref{lem:nth-carries},
and \texttt{bit\_count} supplies $\HW$. These are evaluation
shortcuts, not extra primitives in the reported operation counts.
In particular, three-argument \texttt{pow} avoids constructing the full
power that is one intermediate value of the literal straight-line
program. The value $G_0(x,0)$ is $0$.

\begin{quote}
\begin{verbatim}
def factorial(n):
    f = 1
    for i in range(2, n + 1):
        f *= i
    return f

def G0(x, k):
    if k == 0:
        return 0
    return (x**k - 1) // (x - 1)

def G1(x, m):
    # m is the number of terms; m >= 1
    z = x**m
    num = ((m - 1)*x*z + x) - m*z
    return num // (x - 1)**2

def nu2_binom(upper, lower):
    other = upper - lower
    return (lower.bit_count()
            + other.bit_count()
            - upper.bit_count())

def nth_prime(n):
    C = n*(n + 1)
    F = factorial(n + 1)
    h = (n + 2)**2
    D = 1 << h
    H = 1 << (n + 1)
    B = D*H
    A = B**(C + 1)
    E = A**(n + 1)
    U = G0(B, C + 1)
    mod = E*E
    R = (mod + 1 - pow(A - 1, E*F, mod))//E - F*A
    kappa = (D + F - 1)//F
    mask = D*(H - 1)*G0(B, (C + 1)*(n + 1))
    Q = ((kappa*R*G1(B, C + 1)) & mask)//D
    W = Q % (A - 1)
    Delta = ((A - B + 1)*W) % A - B*B*G0(B, n - 1)
    P = U - 1 - B - (((Delta + (H - 1)*U) & (H*U))//H)
    S = (P*U) % A
    val = nu2_binom(S + (2*H - n)*U, H*U)
    return C + 1 - val

print([nth_prime(n) for n in range(1, 7)])
\end{verbatim}
\end{quote}

Running the program prints
\[
[2, 3, 5, 7, 11, 13],
\]
that is $p_1$ through $p_6$. Modular exponentiation dominates the
running time even for these small inputs.

\paragraph{AI assistance.}
The mathematical ideas are the authors'. OpenAI's GPT~6 Astra was used to
assist in writing proofs and preparing the manuscript.

\small
\setlength{\bibsep}{0.25em}

\end{document}